\documentclass[preprint,11pt]{article}
\usepackage{amssymb}
\usepackage{amsfonts}
\usepackage{mathrsfs}
\usepackage{amssymb,amsfonts,amsmath}
\usepackage{graphicx,graphics,epstopdf}
\usepackage{url}
\graphicspath{{./figures/}}
\usepackage{latexsym}
\usepackage{esint}
\usepackage{times}
\usepackage{multicol,enumerate}
\usepackage{epstopdf}
\usepackage{color}
\usepackage{fullpage}
\usepackage{setspace}
\usepackage[english]{babel}
\usepackage[version=3]{mhchem}
\usepackage{amsmath, amssymb, amsthm}
\usepackage{verbatim, latexsym}
\usepackage{colortbl}
\usepackage{multirow}

\def\Xint#1{\mathchoice
{\XXint\displaystyle\textstyle{#1}}%
{\XXint\textstyle\scriptstyle{#1}}%
{\XXint\scriptstyle\scriptscriptstyle{#1}}%
{\XXint\scriptscriptstyle\scriptscriptstyle{#1}}%
\!\int}
\def\XXint#1#2#3{{\setbox0=\hbox{$#1{#2#3}{\int}$ }
\vcenter{\hbox{$#2#3$ }}\kern-.6\wd0}}

\def\dashint{\Xint-}
\newcommand{\roma}{\mathrm{I}}
\newcommand{\romb}{\mathrm{II}}
\newcommand{\romc}{\mathrm{III}}

\newcommand{\etamaxmin}[1]{\eta_{\text{#1}}}
\newcommand{\si}{\mathrm{S}_i}
\newcommand{\ti}{\mathrm{T}_i}
\newcommand{\hi}{\mathrm{H}_i}

\newcommand{\mc}[1]{\mathcal{#1}}

\newcommand{\dx}{\mathrm{d}x}
\newcommand{\prnt}[1]{\left( #1 \right)}

\newcommand{\norm}[1]{\left\|#1\right\|}
\newcommand{\normHsemi}[2]{\left|#1\right|_{H^{1}\prnt{#2}}}

\newcommand{\normE}[2]{\norm{#1}_{H^{1}_{\kappa}\prnt{#2}}}
\newcommand{\seminormE}[2]{|#1|_{H^{1}_{\kappa}\prnt{#2}}}
\newcommand{\normL}[2]{\norm{#1}_{L^2\prnt{#2}}}
\newcommand{\normLi}[2]{\norm{#1}_{L^2_{{\kappa}^{-1}}\prnt{#2}}}
\newcommand{\normLii}[2]{\norm{#1}_{L^2_{\widetilde{\kappa}^{-1}}\prnt{#2}}}
\newcommand{\normLT}[2]{\norm{#1}_{L^2_{\widetilde{\kappa}}\prnt{#2}}}
\newcommand{\normHp}[3]{\norm{#1}_{H^{#3}\prnt{#2}}}
\newcommand{\innerE}[2]{\langle {#1},{#2}\rangle}

\newcommand{\locV}[2]{V^{\mathrm{H}_{#1}}_{\text{#2}}}
\newcommand{\Cov}{C_{\mathrm{ov}}}
\newcommand{\Cw}{{\rm C}_{\mathrm{weak}}}
\newcommand{\Cpoin}[1]{{\rm C}_{\mathrm{poin}}(#1)}
\newcommand{\locv}[3]{{#1}^{#2}_{\mathrm{#3}}}
\newtheorem{theorem}{Theorem}[section]
\newtheorem{assumption}{Assumption}[section]
\newtheorem{remark}{Remark}[section]
\newtheorem{lemma}{Lemma}[section]

\newtheorem{proposition}{Proposition}[section]
\numberwithin{equation}{section}
\usepackage{soul}
\setstcolor{red}

\title{On the Convergence Rates of GMsFEMs for Heterogeneous Elliptic Problems without Oversampling Techniques}
\author{Guanglian Li\thanks{Department of Mathematics, Imperial College London, London SW7 2AZ,
UK. The work was partially carried out when the author was affiliated with Institut f\"ur Numerische
Simulation and Hausdorff Center for Mathematics, Universit\"at Bonn,
            Wegelerstra{\ss}e 6, D-53115 Bonn, Germany.
             (\texttt{lotusli0707@gmail.com}, \texttt{guanglian.li@imperial.ac.uk}).
             }
}
\begin{document}
\maketitle
\begin{abstract}
This work is concerned with the rigorous analysis on the Generalized Multiscale Finite Element Methods (GMsFEMs) for elliptic problems with high-contrast heterogeneous coefficients. GMsFEMs are popular numerical methods for solving flow problems with heterogeneous high-contrast coefficients, and it has demonstrated extremely promising numerical results for a wide range of applications. However, the mathematical justification of the efficiency of the method is still largely missing.

In this work, we analyze two types of multiscale basis functions, i.e., local spectral basis functions and basis functions of local harmonic extension type, within the GMsFEM framework. These constructions have found many applications in the past few years. We establish their optimal convergence in the energy norm under a very mild assumption that the source term belongs to some weighted $L^2$ space, and without the help of any oversampling technique. Furthermore, we analyze the model order reduction of the local harmonic extension basis and prove its convergence in the energy norm. These theoretical findings shed insights into the mechanism behind the efficiency of the GMsFEMs.

\end{abstract}

\section{Introduction}
The accurate mathematical modeling of many important applications, e.g., composite materials, porous media and reservoir
simulation, calls for elliptic problems with heterogeneous coefficients. In order to adequately describe the intrinsic complex properties in
practical scenarios, the heterogeneous coefficients can have
both multiple inseparable scales and high-contrast. Due to the disparity of scales, the classical numerical treatment becomes prohibitively expensive
and even intractable for many multiscale applications. Nonetheless, motivated by the broad spectrum of practical applications, a large number of multiscale model reduction techniques, e.g., multiscale finite element methods (MsFEMs),
heterogeneous multiscale methods (HMMs), variational multiscale methods, flux norm approach, generalized multiscale
finite element methods (GMsFEMs) and localized orthogonal decomposition (LOD), have been proposed in the literature
\cite{MR1455261,MR1979846,MR1660141,MR2721592, egh12, MR3246801, li2017error} over
the last few decades. They have achieved
great success in the efficient and accurate simulation of heterogeneous problems. Amongst these numerical methods, the GMsFEM \cite{egh12} has
demonstrated extremely promising numerical results for a wide variety of problems, and thus it is becoming
increasingly popular. However, the mathematical understanding of the method remains largely missing, despite numerous
successful empirical evidences. The goal of this work is to provide a mathematical justification, by rigorously
establishing the optimal convergence of the GMsFEMs in the energy norm without any restrictive assumptions or oversampling technique.

We first formulate the heterogeneous elliptic problem. Let $D\subset
\mathbb{R}^d$ ($d=1,2,3$) be an open bounded Lipschitz domain {with a boundary $\partial D$}. Then we seek a function $u\in V:=H^{1}_{0}(D)$ such that
\begin{equation}\label{eqn:pde}
\begin{aligned}
\mathcal{L}u:=-\nabla\cdot(\kappa\nabla u)&=f &&\quad\text{ in }D,\\
u&=0 &&\quad\text{ on } \partial D,
\end{aligned}
\end{equation}
where the force term $f\in L^2(D)$ and the permeability coefficient $\kappa\in L^{\infty}(D)$ with $\alpha\leq\kappa(x)
\leq\beta$ almost everywhere for some lower bound $\alpha>0$ and upper bound $\beta>\alpha$. We denote by $\Lambda:=
\frac{\beta}{\alpha}$ the ratio of these bounds, {which reflects the contrast of the coefficient $\kappa$}. Note that
the existence of multiple scales in the coefficient $\kappa$ rends directly solving Problem \eqref{eqn:pde} challenging, since
resolving the problem to the finest scale would incur huge computational cost.

The goal of the GMsFEM is to efficiently capture the large-scale behavior of the solution $u$ locally without
resolving all the microscale features within. To realize this desirable property, we first discretize the computational
domain $D$ into a coarse mesh $\mathcal{T}^H$. Over $\mathcal{T}^H$, we define the classical multiscale
basis functions $\{\chi_i\}_{i=1}^{N}$, with $N$ being the total number of coarse nodes. Let $\omega_i:=\text{supp}
(\chi_i)$ be the support of $\chi_i$, which is often called a local coarse neighborhood below. To
accurately approximate the local solution $u|_{\omega_i}$ (restricted to $\omega_i$), we construct
a local approximation space. In practice, two types of local multiscale spaces are frequently employed:
local spectral space ($V_{\text{off}}^{\si, \ell_i^{\roma}}$, of dimension $\ell_i^{\roma}$) and local harmonic space
$V_{\text{snap}}^{\hi}$. The dimensionality of the local harmonic space $V_{\text{snap}}^{\hi}$ is problem-dependent, and it can be
extremely large when the microscale within the coefficient $\kappa$ tends to zero. Hence, a further local model reduction based
on proper orthogonal decomposition (POD) in $V_{\text{snap}}^{\hi}$ is often employed. We denote the corresponding
local POD space of rank $\ell_i$ by $V_{\text{off}}^{\hi, \ell_i}$. In sum, in practice, we can have three types of local
multiscale spaces at our disposal: $V_{\text{off}}^{\si, \ell_i}$, $V_{\text{snap}}^{\hi}$ and $V_{\text{off}}^{\hi, \ell_i}$ on
$\omega_i$. These basis functions are then used in the standard finite element framework, e.g., continuous
Galerkin formulation, for constructing
a global approximate solution.

One crucial part in the local spectral basis construction is to include local spectral basis functions ($V_{\text{off}}^{\ti, \ell_i^{\romb}}$, of dimension $\ell_i^{\romb}$) governed by Steklov eigenvalue problems \cite{MR2770439}, which was first applied to the context of the GMsFEMs in \cite{MR3277208}, to the best of our knowledge.
This was motivated by the decomposition of the local solution $u|_{\omega_i}$ into the sum of three components, cf. \eqref{eq:decomp}, where the first two components can be approximated efficiently by the local spectral space $V_{\text{off}}^{\si, \ell_i^{\roma}}$ and $V_{\text{off}}^{\ti, \ell_i^{\romb}}$, respectively, and the third component is of rank one and can be obtained by solving one local problem.

The good approximation property of these local multiscale spaces to the solution $u|_{\omega_i}$ of problem
\eqref{eqn:pde} is critical to ensure the accuracy and efficiency of the GMsFEM. We shall present relevant
approximation error results for the preceding three types of multiscale basis functions in Proposition \ref{prop:projection}, Lemma \ref{lemma:u2}, Lemma \ref{lem:energyHA} and Lemma
\ref{lem:5.2}. It is worth pointing out that the proof of Proposition \ref{prop:projection} relies crucially on the expansion of the
source term $f$ in terms of the local spectral basis function in Lemma \ref{lem:assF}. Thus the argument differs
substantially from the typical argument for such analysis that employs the oversampling argument together with a Cacciopoli type
inequality \cite{babuska2011optimal,eglp13}, and it is of independent interest by itself.

The proof to Lemma \ref{lemma:u2} is very critical. It relies essentially on the transposition method \cite{MR0350177},
which bounds the weighted $L^2$ error estimate in the domain by the boundary error estimate, since the latter can be
obtained straightforwardly. Most importantly, the involved constant is independent of the contrast in the
coefficient $\kappa$. This result is presented in Theorem \ref{lem:very-weak}.

To establish Lemmas \ref{lem:energyHA} and \ref{lem:5.2}, we make one mild assumption on the
geometry of the coefficient, cf. Assumption \ref{ass:coeff}, which enables the use
of the weighted Friedrichs inequality in the  proof. 
In addition, since the local multiscale basis functions in  $V_{\text{off}}^{\hi, \ell_i}$ are $\kappa$-harmonic
and since the weighted $L^2(\omega_i)$ error estimate can be obtained directly from the POD, cf. Lemma \ref{lem:5.1}, we employ a
Cacciopoli type inequality \cite{MR717034} to prove Lemma \ref{lem:5.2}. Note that our analysis does not
exploit the oversampling strategy, which has played a crucial role for proving energy error estimates
in all existing works \cite{babuska2011optimal,eglp13,MR3246801,chung2017constraint}.

Together with the conforming Galerkin formulation and the partition of unity functions $\{\chi_i\}_{i=1}^N$
on the local domains $\{\omega_i\}_{i=1}^{N}$, we obtain three types of multiscale methods to solve
problem \eqref{eqn:pde}, cf. \eqref{cgvarform_spectral}--\eqref{cgvarform_pod}. Their energy error estimates
are presented in Propositions \ref{prop:Finalspectral}, \ref{prop:FinalSnap} and \ref{prop:Finalpod},
respectively. Specifically, their convergence rates are precisely characterized by the eigenvalues $\lambda_{\ell_i^{\roma}}^{\si}$, $\lambda_{\ell_i^{\romb}}^{\ti}$,
$\lambda_{\ell_i}^{\hi}$ and the coarse mesh size $H$ (see Section \ref{sec:error} for the definitions
of the eigenvalue problems). Thus, the decay/growth behavior of these eigenvalues plays an extremely
important role in determining the convergence rates, which, however, is beyond the scope of the present work. We refer
 readers to the works \cite{babuska2011optimal,li2017low} for results along this line.

Last, we put our contributions into the context. The local spectral estimates in the energy norm in
Proposition \ref{prop:projection} and Lemma \ref{lemma:u2} represent the state-of-art result in the sense that no restrictive
assumption on the problem data is made. Furthermore, we prove the convergence without the help
of the oversampling strategy in the analysis, which has played a crucial role in all existing studies
\cite{babuska2011optimal,EFENDIEV2011937,eglp13,chung2017constraint}. In practice, avoiding oversampling
strategy allows saving computational cost, and this also corroborates well empirical observations \cite{EFENDIEV2011937}.
Due to the local estimates in Proposition \ref{prop:projection} and Lemma \ref{lemma:u2},
we are able to derive a global estimate in Proposition \ref{prop:Finalspectral} that is the much needed
results for analyzing many multiscale methods \cite{MR1660141, MR2721592, MR3246801,li2017error}, cf. Remark
\ref{rem:spectral}. Recently Chung et al \cite{chung2017constraint} proved some convergence estimates
in a similar spirit to Proposition \ref{prop:projection}, by adapting the LOD technique \cite{MR3246801}.
Our result greatly simplifies the analysis and improves their result \cite{chung2017constraint} by avoiding
the oversampling. To the best of our knowledge, there is no known convergence estimate for
either the local harmonic space or the local POD space, and the results presented in Propositions
\ref{prop:FinalSnap} and \ref{prop:Finalpod} are the first such results.

The remainder of this paper is organized as follows. We formulate the heterogeneous problem in Section \ref{sec:pre},
and describe the main idea of the GMsFEM. We present in Section \ref{cgdgmsfem} the
construction of local multiscale spaces, harmonic extension space and discrete POD. Based upon them, we present three type of
global multiscale spaces. Together with the canonical conforming Galerkin formulation, we
obtain three type of numerical methods to approximate Problem \eqref{eqn:pde} in \eqref{cgvarform_spectral}
to \eqref{cgvarform_pod}. The error estimates of these multiscale methods are presented in Section \ref{sec:error},
which represent the main contributions of this paper. Finally, we conclude the paper with concluding remarks in
Section \ref{sec:conclusion}. We establish the regularity result of the elliptic problem with very rough boundary data in an appendix.

\section{Preliminary}\label{sec:pre}
Now we present basic facts related to Problem \eqref{eqn:pde} and briefly describe the GMsFEM (and also to fix the notation).
Let the space $V:=H^{1}_{0}(D)$ be equipped with the (weighted) inner product
\begin{align*}
\innerE{v_1}{v_2}_{D}=:a(v_1,v_2):=\int_{D}\kappa\nabla v_1\cdot\nabla v_2\;\dx\quad \text{ for all } v_1, v_2\in V,
\end{align*}
and the associated energy norm
\begin{align*}
\seminormE{v}{D}^2:=\innerE{v}{v}_{D}\quad \text{ for all } v\in V.
\end{align*}
We denote by $W:=L^2(D)$ equipped with the usual norm $\normL{\cdot}{D}$ {and inner product $(\cdot,\cdot)_{D}$}.

The weak formulation for problem \eqref{eqn:pde} is to find $u\in V$ such that
\begin{align}\label{eqn:weakform}
a(u,v)=(f,v)_{D} \quad \text{for all
} v\in V.
\end{align}
The Lax-Milgram theorem implies the well-posedness of problem \eqref{eqn:weakform}.

To discretize problem \eqref{eqn:pde}, we first introduce fine and coarse grids.
Let $\mathcal{T}^H$ be a regular partition of the domain $D$ into
finite elements (triangles, quadrilaterals, tetrahedra, etc.) with a mesh size $H$. We refer to
this partition as coarse grids, and accordingly the course elements. Then each coarse element is further partitioned
into a union of connected fine grid blocks. The fine-grid partition is denoted by
$\mathcal{T}^h$ with $h$ being its mesh size. Over $\mathcal{T}^h$, let $V_h$ be the conforming piecewise
linear finite element space:
\[
V_h:=\{v\in \mathcal{C}: V|_{T}\in \mathcal{P}_{1} \text{ for all } T\in \mathcal{T}^h\},
\]
where $\mathcal{P}_1$ denotes the space of linear polynomials. Then the fine-scale solution $u_h\in V_h$ satisfies
\begin{align}\label{eqn:weakform_h}
a(u_h,v_h)=(f,v_h)_{D} \quad \text{ for all } v_h\in V_h.
\end{align}
The Galerkin orthogonality implies the following optimal estimate in the energy norm:
\begin{align}\label{eq:fineApriori}
\seminormE{u-u_h}{D}\leq \min\limits_{v_h\in V_h}\seminormE{u-v_h}{D}.
\end{align}
The fine-scale solution $u_h$ will serve as a reference solution in multiscale methods. Note that due to the presence of multiple scales in the coefficient $\kappa$, the fine-scale mesh size $h$ should be commensurate with the smallest scale and thus it can be very small in order to obtain an accurate solution. This necessarily involves huge computational complexity, and more efficient methods are in great demand.


In this work, we are concerned with flow problems with high-contrast heterogeneous coefficients,
which involve multiscale permeability fields, e.g., permeability fields with vugs and faults, and
furthermore, can be parameter-dependent, e.g., viscosity. Under such scenario, the computation of the
fine-scale solution $u_h$ is vulnerable to high computational
complexity, and one has to resort to multiscale methods. The GMsFEM has been extremely successful
for solving multiscale flow problems, which we briefly recap below. 

The GMsFEM aims at solving Problem \eqref{eqn:pde} on the coarse mesh $\mathcal{T}^{H}$
cheaply, which, meanwhile, maintains a certain accuracy compared to the fine-scale solution $u_h$. To describe the
GMsFEM, we need a few notation. The vertices of $\mathcal{T}^H$
are denoted by $\{O_i\}_{i=1}^{N}$, with $N$ being the total number of coarse nodes.
The coarse neighborhood associated with the node $O_i$ is denoted by
\begin{equation} \label{neighborhood}
\omega_i:=\bigcup\{ K_j\in\mathcal{T}^H: ~~~ O_i\in \overline{K}_j\}.
\end{equation}
The overlap constant $\Cov$ is defined by
\begin{align}\label{eq:overlap}
\Cov:=\max\limits_{K\in \mathcal{T}^{H}}\#\{O_i: K\subset\omega_i \text{ for } i=1,2,\cdots,N\}.
\end{align}
We refer to Figure~\ref{schematic} for an illustration of neighborhoods and elements subordinated to the coarse
discretization $\mathcal{T}^H$. Throughout, we use $\omega_i$ to denote a coarse neighborhood.

\begin{figure}[htb]
  \centering
  \includegraphics[width=0.65\textwidth]{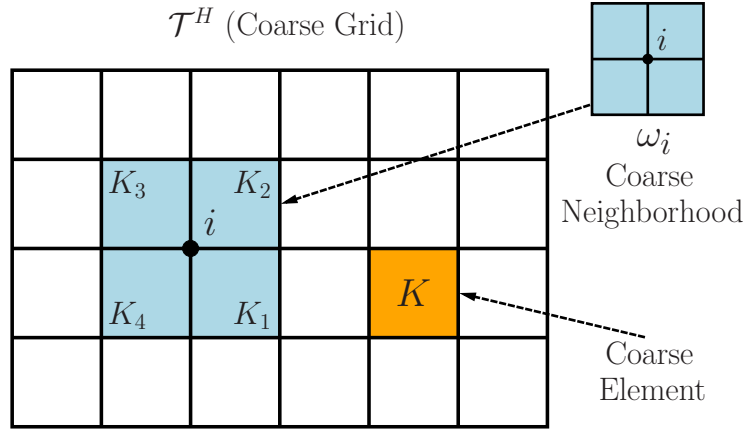}
  \caption{Illustration of a coarse neighborhood and coarse element with an overlapping constant $\Cov=4$.}
  \label{schematic}
\end{figure}

Next, we outline the GMsFEM with a continuous Galerkin (CG) formulation; see Section \ref{cgdgmsfem} for details. We denote by $\omega_i$
the support of the multiscale basis functions. These basis functions are denoted by $\psi_k^{\omega_i}$ for
$k=1,\cdots,\ell_i$ for some $\ell_i\in \mathbb{N}_{+}$, which is the number of local basis functions associated with $\omega_i$. Throughout,
the superscript $i$ denotes the $i$-th coarse node or coarse neighborhood $\omega_i$.
Generally, the GMsFEM utilizes multiple basis functions per coarse neighborhood $\omega_i$,
and the index $k$ represents the numbering of these basis functions.
In turn, the CG multiscale solution $u_{\text{ms}}$ is sought as $u_{\text{ms}}(x)=\sum_{i,k} c_{k}^i \psi_{k}^{\omega_i}(x)$.
Once the basis functions $\psi_k^{\omega_i}$ are identified, the CG global coupling is given through the variational form
\begin{equation}
\label{eq:globalG} a(u_{\text{ms}},v)=(f,v), \quad \text{for all} \, \, v\in
V_{\text{off}},
\end{equation}
where $V_{\text{off}}$ denotes the finite element space spanned by these basis functions.

We conclude the section with the following assumption on $\Omega$ and $\kappa$.
\begin{assumption}[Structure of $D$ and $\kappa$]\label{ass:coeff}
Let $D$ be a domain with a $C^{1,\alpha}$ $(0<\alpha<1)$ boundary $\partial D$,
and $\{D_i\}_{i=1}^m\subset D$ be $m$ pairwise disjoint strictly convex open subsets, {each with a $C^{1,\alpha}$ boundary
$\Gamma_i:=\partial D_i$,} and denote $D_0=D\backslash \overline{\cup_{i=1}^{m} D_i}$. 
Let the permeability coefficient $\kappa$ be piecewise regular function defined by
\begin{equation}
\kappa=\left\{
\begin{aligned}
&\eta_{i}(x) &\text{ in } D_{i},\\
&1 &\text{ in }D_0.
\end{aligned}
\right.
\end{equation}
Here $\eta_i\in C^{\mu}(\bar{D_i})$ with $\mu\in (0,1)$ for $i=1,\cdots,m$. Denote  $\etamaxmin{min}:=\min_{i}\{\eta_i\}\geq 1$ and $\etamaxmin{max}:=\max_{i}\{\eta_i\}$.
\end{assumption}

Under Assumption \ref{ass:coeff}, the coefficient $\kappa$ is $\Gamma$-{\em quasi-monotone} on each coarse neighborhood $\omega_i$ and the global domain $D$
(see \cite[Definition 2.6]{pechstein2012weighted} for the precise definition) with either $\Gamma:=\partial \omega_i$
or $\Gamma:=\partial D$. Then the following weighted Friedrichs inequality \cite[Theorem 2.7]{pechstein2012weighted} holds.
\begin{theorem}[Weighted Friedrichs inequality]\label{thm:friedrichs}
Let $\text{diam}(D)$ be the diameter of the bounded domain $D$ and $\omega_i\subset D$. Define
\begin{align}
\Cpoin{\omega_i}&:=H^{-2}\max\limits_{w\in H^1_0(\omega_i)}\frac{\int_{\omega_i}{\kappa}w^2\dx}{\int_{\omega_i}\kappa|\nabla w|^2\dx},\label{eq:poinConstant}\\
\Cpoin{D}&:=\text{diam}(D)^{-2}\max\limits_{w\in H^1_0(D)}\frac{\int_{D}{\kappa}w^2\dx}{\int_{D}\kappa|\nabla w|^2\dx}.\label{eq:poinConstantG}
\end{align}
Then the positive constants $\Cpoin{\omega_i}$ and $\Cpoin{D}$ are independent of the contrast of $\kappa$.
\end{theorem}
\begin{remark}
Below we only require that the constants $\Cpoin{\omega_i}$ and $\Cpoin{D}$ be independent
of the contrast in $\kappa$. Assumption \ref{ass:coeff} is one sufficient condition to ensure this,
and it can be relaxed \cite{pechstein2012weighted}.
\end{remark}



\section{CG-based GMsFEM for high-contrast flow problems}
\label{cgdgmsfem}
In this section, we present the local spectral basis functions, local harmonic extension
basis functions and POD, and the global weak formulation based on these local multiscale basis functions.

\subsection{Local multiscale basis functions}
\label{locbasis}
First we present two principled approaches for constructing local multiscale functions: local spectral
bases and local harmonic extension bases, which represent the two main approaches within the GMsFEM framework.
The constructions are carried out on each coarse neighborhood $\omega_i$ with $i=1,2,\cdots,N$, and can be carried out in
parallel, if desired. Since the dimensionality of the local harmonic extension bases is problem-dependent and inversely proportional
to the smallest scale in $\kappa$, in practice,  we often perform an ``optimal'' local model order reduction based
on POD to further reduce the complexity at the online stage.

Before presenting the constructions, we first introduce some useful function spaces, which will play an important role in the analysis below.
Let $L^2_{\widetilde{\kappa}}(\omega_i)$ and $H^1_{\kappa}(\omega_i)$ be Hilbert spaces
with their inner products and norms defined respectively by
\begin{alignat*}{3}
(w_1,w_2)_{i}&:=\int_{\omega_i}\widetilde{\kappa}w_1\cdot w_2\;\dx&&\|{w_1}\|_{L^2_{\widetilde{\kappa}}(\omega_i)}^2:=(w_1,w_1)_{i}&\ \ \text{ for }w_1, w_2\in L^2_{\widetilde{\kappa}}(\omega_i),\\
\innerE{v_1}{v_2}_{i}&:=\int_{\omega_i}{\kappa}\nabla v_1\cdot \nabla v_2\;\dx \quad&&\normE{v_1}{\omega_i}^2:=(v_1,v_2)_i+\innerE{v_1}{v_1}_{i}&\text{ for } v_1,v_2\in H^1_{\kappa}(\omega_i).
\end{alignat*}
Next we define two subspaces $W_i\subset L^2_{ \widetilde{\kappa}}(\omega_i)$ and $V_i\subset H^1_{\kappa}(\omega_i)$ of codimension one by
\[
W_i:=\{v\in L^2_{ \widetilde{\kappa}}(\omega_i):\int_{\omega_i}\widetilde{\kappa}v\;\dx=0\}
\quad \mbox{and}\quad
V_i:=\{v\in H^1_{\kappa}(\omega_i):\int_{\omega_i}\widetilde{\kappa}v\;\dx=0\}.
\]
Furthermore, we introduce the following weighted Sobolev spaces:
\begin{align*}
L_{{\kappa}^{-1}}^{2}(\omega_i):=&\Big\{w:\|w\|_{L^2_{\kappa^{-1}(\omega_i)}}^2:=\int_{\omega_i}{\kappa}^{-1} w^2\dx<\infty  \Big\},\\
H_{\kappa,0}^{1}(\omega_i):=&\Big\{w: w|_{\partial{\omega_i}}=0\text{ s.t. }\seminormE{w}{\omega_i}^2:=\int_{\omega_i}\kappa |\nabla w|^2\dx<\infty  \Big\}.
\end{align*}
Similarly, we define the following weighted Sobolev spaces with their associated norms: $(L_{\widetilde{\kappa}^{-1}}^{2}(\omega_i),\|\cdot\|_{L^2_{\widetilde{\kappa}^{-1}(\omega_i)}})$, $(L_{{\kappa}^{-1}}^{2}(D),\|\cdot\|_{L^2_{\kappa^{-1}(D)}})$ and $(L_{\widetilde{\kappa}^{-1}}^{2}(D),\|\cdot\|_{L^2_{\widetilde{\kappa}^{-1}(D)}})$. The nonnegative weights $\widetilde{\kappa}$ and $\widetilde{\kappa}^{-1}$ will be defined in \eqref{defn:tildeKappa} and \eqref{eq:inv-tildeKappa} below, respectively.

Throughout, the superscripts $\si$, $\ti$ and $\hi$ are associated to the local
spectral spaces and local harmonic space on $\omega_i$, respectively. Below we
describe the construction of local multiscale basis functions on $\omega_i$.

\subsubsection*{Local spectral bases I}

To define the local spectral bases on $\omega_i$, we first introduce a local elliptic operator $\mathcal{L}_i$ on $\omega_i$ by
\begin{align}\label{eq:Li}
  \left\{ \begin{aligned}
          \mathcal{L}_i v&:=-\nabla\cdot(\kappa\nabla v)\quad \mbox{in }\omega_i,\\
          \kappa\frac{\partial v}{\partial n}&=0\quad \mbox{on }\partial\omega_i.
  \end{aligned}\right.
\end{align}
The Lax-Milgram theorem implies the well-posedness of the operator $\mathcal{L}_i:V_i\to V_i^*$,
the dual space $V_i^{*}$ of $V_i$. 
Then the spectral problem can be formulated in terms of
$\mathcal{L}_i$, i.e., to seek $(\lambda_{j}^{\si}, v_{j}^{\si})\in \mathbb{R}\times V_i$ such that
\begin{alignat}{2}\label{eq:spectral}
\mathcal{L}_i v_{j}^{\si}  &= \widetilde{\kappa}\lambda_{j}^{\si} v_{j}^{\si}
\quad &&\text{in} \, \, \, \omega_i,\\
\kappa\frac{\partial}{\partial n}v_{j}^{\si}&=0&&\text{ on } \partial \omega_i,\nonumber
\end{alignat}
where the parameter $\widetilde\kappa$ is defined by
\begin{equation}\label{defn:tildeKappa}
\widetilde{\kappa} =H^2 \kappa \sum_{i=1}^{N}  | \nabla \chi_i |^2,
\end{equation}
with the multiscale function $\chi_i$ to be defined in \eqref{pou} below. Note that the use of $\widetilde{\kappa}$ in the local spectral problem \eqref{eq:spectral} instead of $\kappa$ is due to numerical consideration \cite{EFENDIEV2011937}. Furthermore, let $\widetilde{\kappa}^{-1}$ be defined by
\begin{equation}\label{eq:inv-tildeKappa}
\widetilde{\kappa}^{-1}(x)=
\left\{
\begin{aligned}
&\widetilde{\kappa}^{-1}, \quad &&\text{ when } \widetilde{\kappa}(x)\ne 0\\
&1, \quad &&\text{ otherwise }.
\end{aligned}
\right.
\end{equation}

\begin{remark}
Generally, one cannot preclude the existence of critical points from
the multiscale basis functions $\chi_i$
\cite{MR1289138,alberti2017critical}. In the two-dimensional case, it was proved
that there are at most a finite number of isolated critical points.
To simplify our presentation, we will assume $|D\cap\{\widetilde{\kappa}=0\}|=0$.
\end{remark}

The next result gives the eigenvalue behavior of the local spectral problem \eqref{eq:spectral}.
\begin{theorem}\label{lem:eigenvalue-blowup}
Let $\{(\lambda_j^{\si},v_j^{\si})\}_{j=1}^{\infty}$ be the eigenvalues and the corresponding normalized eigenfunctions in $W_i$ to the spectral problem \eqref{eq:spectral} listed according to their algebraic
multiplicities and the eigenvalues are ordered nondecreasingly. There holds
\begin{align}\label{eq:spectral_eigenvalue}
\lambda_j^{\si}\to \infty\quad \text{ as } j\to \infty.
\end{align}
\end{theorem}
To prove Theorem \ref{lem:eigenvalue-blowup}, we need a few notation. Let $\mathcal{S}_i:=\mathcal{L}_i^{-1}: V^{*}_i\to V_i$
be the inverse of the elliptic operator $\mathcal{L}_i$.
Denote $T:W_i\to L^2_{\widetilde{\kappa}^{-1}}(\omega_i)$ to be the multiplication operator defined by
\begin{align}\label{eq:T}
Tv:=\widetilde{\kappa}v \quad\text{ for all }\quad v\in W_i.
\end{align}
One can show by definition directly that $T$ is a bounded operator with unit norm. Moreover, there holds
\[
\int_{\omega_i}Tv\;\dx=0 \quad\text{ for all }v\in W_i.
\]
Thus the range of $T$, $\mathcal{R}(T)$, is a subspace in $L^2_{\widetilde{\kappa}^{-1}}(\omega_i)$ with codimension one, and we have
\begin{align}\label{eq:R(T)}
\mathcal{R}(T)\hookrightarrow V_i^{*}.
\end{align}

For the proof of Theorem \ref{lem:eigenvalue-blowup}, we need the following compact embedding result.
\begin{lemma}\label{lem:embedding}
$V_i$ is compactly embedded into $W_i$, i.e.,
$V_i \hookrightarrow\hookrightarrow W_i.$
\end{lemma}
\begin{proof}
By Remark \ref{rem:chi}, the uniform boundedness of $\kappa$, the definition of $\widetilde{\kappa}$
and the overlapping condition \eqref{eq:overlap}, we obtain the boundedness of $\tilde{\kappa}$, i.e.,
\begin{align}\label{eq:upper_tilde}
\|\widetilde{\kappa}\|_{L^{\infty}(D)}\leq C_{\text{ov}}(HC_{0})^2\kappa\leq C_{\text{ov}}(HC_{0})^2\beta.
\end{align}
Hence, there holds the following embedding inequalities:
\[
L^2_{\widetilde{\kappa}^{-1}}(\omega_i)\hookrightarrow L^2(\omega_i)\hookrightarrow L^2_{\widetilde{\kappa}}(\omega_i).
\]
This, the classical Sobolev embedding \cite{adams2003sobolev} and boundedness of $\kappa$ imply
the compactness of the embedding $V_i\hookrightarrow\hookrightarrow L^2(\omega_i)$ and thus, we
finally arrive at $V_i \hookrightarrow\hookrightarrow W_i$. This completes the proof.
\end{proof}
\begin{proof}[Proof of Theorem \ref{lem:eigenvalue-blowup}]
By \eqref{eq:R(T)}, the multiplication operator $T: W_i\to V_i^*$ is bounded.
Similarly, the operator $\mathcal{S}_i:V_i^*\to W_i$ is compact, in view of Lemma \ref{lem:embedding}.
Let $\widetilde{\mathcal{S}}_i:=\mathcal{S}_i T$.
Then the operator $\widetilde{\mathcal{S}}_i:W_i\to W_i$ is nonnegative and {compact}.
Now we claim that $\widetilde{\mathcal{S}}_i$ is self-adjoint on $W_i$.
Indeed, for all $v,w\in W_i$, we have
\begin{align*}
(\widetilde{\mathcal{S}}_i v, w)_i&=(\mathcal{S}_i Tv, w)_i=\int_{\omega_i}\widetilde{\kappa}\mathcal{L}_i^{-1}(\widetilde{\kappa}v) w\;\dx\\
&=\int_{\omega_i}\mathcal{L}_i^{-1}(\widetilde{\kappa}v) (\widetilde{\kappa}w)\;\dx\\
&=(v,(\mathcal{S}_i T)w)_i=(v,\widetilde{\mathcal{S}}_iw)_i,
\end{align*}
where we have used the weak formulation for \eqref{eq:Li} to deduce
$
\int_{\omega_i}\mathcal{L}_i^{-1}(\widetilde{\kappa}v) (\widetilde{\kappa}w)\dx=\int_{\omega_i}(\widetilde{\kappa}v)\mathcal{L}_i^{-1} (\widetilde{\kappa}w)\dx.
$
By the standard spectral theory for compact operators \cite{yosida78}, it has at most countably many discrete eigenvalues, with zero being
the only accumulation point, and each nonzero eigenvalue has only finite multiplicity.
Noting that $\big\{\big((\lambda_j^{\si})^{-1},
v_j^{\si}\big)\big\}_{j=1}^{\infty}$ are the eigenpairs of $\widetilde{\mathcal{S}}_i$ completes the proof.
\end{proof}
Furthermore, by the construction, the eigenfunctions $\{ v_j^{\si}\}_{j=1}^{\infty}$
form a complete orthonormal bases (CONB) in $W_i$, and $\{\sqrt{\lambda_j^{\si}+1}{v_j^{\si}}\}_{j=1}^{\infty}$
form a CONB in $V_i$. Further, we have $L^2_{\widetilde{\kappa}}(\omega_i)=W_i\oplus \{1\}$.
Hence, $\{ v_j^{\si}\}_{j=1}^{\infty}\oplus \{1\}$ is a complete orthogonal bases in
$L^2_{\widetilde{\kappa}}(\omega_i)$ [Chapters 4 and 5]\cite{laugesen}\footnote{We thank Richard S. Laugesen (University of Illinois, Urbana-Champaign) for clarifying the convergence in $H^1_{\kappa}(\omega_i)$.}.
\begin{lemma}\label{lem:L2Inv}
The series $\{ \widetilde{\kappa} v_j^{\si}\}_{j=1}^{\infty}\oplus \{\widetilde{\kappa}\}$ forms a complete
orthogonal bases in $L_{\widetilde{\kappa}^{-1}}^2(\omega_i)$.
\end{lemma}
\begin{proof}
First, we show that $\{ \widetilde{\kappa} v_j^{\si}\}_{j=1}^{\infty}\oplus \{\widetilde{\kappa}\}$ are orthogonal in $L_{\widetilde{\kappa}^{-1}}^2(\omega_i)$.
Indeed, by definition, we deduce that for all $j\in \mathbb{N}_{+}$
\begin{align*}
\int_{\omega_i}\widetilde{\kappa}^{-1}\widetilde{\kappa}\cdot  \widetilde{\kappa} v_j^{\si}\dx
=\int_{\omega_i}\widetilde{\kappa} v_j^{\si}\dx=(v_j^{\si},1)_i=0.
\end{align*}
Meanwhile, for all $j,k\in \mathbb{N}_{+}$, there holds
\begin{align*}
\int_{\omega_i}\widetilde{\kappa}^{-1}\widetilde{\kappa}v_k^{\si}\cdot  \widetilde{\kappa} v_j^{\si}\dx
=\int_{\omega_i}\widetilde{\kappa} v_j^{\si}\cdot v_k^{\si}\dx=(v_j^{\si},v_k^{\si})_i=\delta_{j,k}.
\end{align*}
Next we show that $\{ \widetilde{\kappa} v_j^{\si}\}_{j=1}^{\infty}\oplus \{\widetilde{\kappa}\}$ are complete in $L_{\widetilde{\kappa}^{-1}}^2(\omega_i)$.
Actually, for any $v\in L_{\widetilde{\kappa}^{-1}}^2(\omega_i)$ such that
\begin{equation}\label{eq:9}
\begin{aligned}
\int_{\omega_i}\widetilde{\kappa}^{-1}v\cdot  \widetilde{\kappa}\dx=0\quad
\text{and }\quad \forall j\in \mathbb{N}_{+}:
\int_{\omega_i}\widetilde{\kappa}^{-1}v\cdot  \widetilde{\kappa} v_j^{\si}\dx=0,
\end{aligned}
\end{equation}
we deduce directly from definition that
\begin{align*}
\int_{\omega_i}\widetilde{\kappa}(\widetilde{\kappa}^{-1}v)^2\dx
=\int_{\omega_i\cap\{\widetilde{\kappa}\ne 0\}}\widetilde{\kappa}^{-1}v^2\dx<\infty.
\end{align*}
This implies that $\widetilde{\kappa}^{-1}v\in L^2_{\widetilde{\kappa}}(\omega_i)$.
Furthermore, \eqref{eq:9} indicates that $\widetilde{\kappa}^{-1}v$
is orthogonal to a set of complete orthogonal basis functions $\{ v_j^{\si}\}_{j=1}^{\infty}\oplus \{1\}$ in $L^2_{\widetilde{\kappa}}(\omega_i)$. Therefore, $v=0$, which completes the proof.
\end{proof}
\begin{remark}\label{rem:dual}
Since $L^2_{\widetilde{\kappa}^{-1}}(\omega_i)$ is a Hilbert space, we can identify its dual with itself, and there exists an isometry between $L^2_{\widetilde{\kappa}}(\omega_i)$ and $L^2_{\widetilde{\kappa}^{-1}}(\omega_i)$, e.g., the operator $T$ in \eqref{eq:T}. We identify  $L^2_{\widetilde{\kappa}}(\omega_i)$ as the dual of $L^2_{\widetilde{\kappa}^{-1}}(\omega_i)$.
\end{remark}

Now we define the local spectral basis functions on $\omega_i$ for all
$i=1,\cdots, N$. Let $\ell_i^{\roma}\in \mathbb{N}_{+}$ be a prespecified number, denoting the number of local
basis functions associated with $\omega_i$. We take the eigenfunctions corresponding to the first
$(\ell_i^{\roma}-1)$ smallest eigenvalues for problem \eqref{eq:spectral} in addition to the
kernel of the elliptic operator $\mathcal{L}_i$, namely, $\{1\}$, to construct the local spectral offline space:
\[
V_{\text{off}}^{\text{S}_i,\ell_i^{\roma}}= \text{span}\{ v_{j}^{\si}:~~ 1\leq j <\ell_i^{\roma}\}\oplus \{1\}.
\]
Then $\dim(V_{\text{off}}^{\text{S}_i,\ell_i^{\roma}})=\ell_i^{\roma}$. The choice of the truncation number
$\ell_i^{\roma}\in \mathbb{N}_{+}$ has to be determined by the eigenvalue decay rate or the presence of
spectral gap. The space $V_{\text{off}}^{\text{S}_i,\ell_i^{\roma}}$ allows defining a finite-rank projection
operator $\mathcal{P}^{\si,\ell_i^{\roma}}: L^2_{\widetilde{\kappa}}(\omega_i)\to V_{\text{off}}^{\text{S}_i,
\ell_i^{\roma}}$ by (with the constant $c_0=\big(\int_{\omega_i}\widetilde{\kappa}  \dx \big)^{-1}$):
\begin{align}\label{eq:FR_spec}
\mathcal{P}^{\si,\ell_i^{\roma}}v=c_0(v,1)_i+\sum\limits_{j=1}^{\ell_i^{\roma}-1}(v,v_j^{\si})_i v_j^{\si}\ \ \text{ for all } v\in L_{\tilde \kappa}^2(\omega_i).
\end{align}
The operator $\mathcal{P}^{\si,\ell_i^{\roma}}$ will play a role in the convergence analysis.

\subsubsection*{Local Steklov eigenvalue problem II}
The local Steklov eigenvalue problem can be formulated as to seeking $(\lambda_{j}^{\ti}, v_{j}^{\ti})\in \mathbb{R}\times H^1_{\kappa}(\omega_i)$ such that
\begin{alignat}{2}\label{eq:steklov}
-\nabla\cdot(\kappa\nabla v_{j}^{\ti})  &= 0 &&\quad\text{in} \, \, \, \omega_i,\\
\kappa\frac{\partial}{\partial n}v_{j}^{\ti}&=\lambda_{j}^{\ti} v_{j}^{\ti}&&\quad\text{ on } \partial \omega_i.\nonumber
\end{alignat}
It is well known that the spectrals of the Steklov eigenvalue problem blow up \cite{MR2770439}:
\begin{theorem}\label{lem:steklov-blowup}
Let $\{(\lambda_j^{\ti},v_j^{\ti})\}_{j=1}^{\infty}$ be the eigenvalues and the corresponding normalized eigenfunctions in $L^2(\partial\omega_i)$ to the spectral problem \eqref{eq:steklov} listed according to their algebraic
multiplicities and the eigenvalues are ordered nondecreasingly. There holds
\begin{align*}
\lambda_j^{\ti}\to \infty\quad \text{ as } j\to \infty.
\end{align*}
\end{theorem}
Note that $\lambda_1^{\ti}=0$ and $v_1^{\ti}$ is a constant. Furthermore, the series $\big\{v_j^{\ti}\big\}_{j=1}^{\infty}$ forms a complete orthonormal bases in $L^2(\partial\omega_i)$. Below we use the notation $(\cdot,\cdot)_{\partial\omega_i}$ to denote the inner product on $L^2(\partial\omega_i)$.
Similarly, we define a local spectral space of dimension $\ell_i^{\romb}$ and the associated $\ell_i^{\romb}$-rank projection operator:
\begin{align}
V_{\text{off}}^{\text{T}_i,\ell_i^{\romb}}&= \text{span}\{ v_{j}^{\ti}:~~ 1\leq j \leq\ell_i^{\romb}\},\nonumber\\
\mathcal{P}^{\ti,\ell_i^{\romb}}v
&=\sum\limits_{j=1}^{\ell_i^{\romb}}( v,v_j^{\ti})_{\partial\omega_i} v_j^{\ti}\ \ \text{ for all } v\in L^2(\partial\omega_i).\label{eq:steklov_spec}
\end{align}
In addition to these local spectral basis functions defined in
Problems \eqref{eq:spectral} and \eqref{eq:steklov}, we need one more local basis function defined by the following local problem:
\begin{equation}\label{eq:1-basis}
\left\{
\begin{aligned}
-\nabla\cdot(\kappa\nabla v^{i})&=\frac{\widetilde{\kappa}}{\int_{\omega_i}\widetilde{\kappa}\dx} \quad&&\text{ in } \omega_i,\\
-\kappa\frac{\partial v^{i}}{\partial n}&=|\partial\omega_i|^{-1}\quad&&\text{ on }\partial \omega_i.
\end{aligned}
\right.
\end{equation}
Note that the approximation property of $V_{\text{off}}^{\text{S}_i,\ell_i^{\roma}}$, $V_{\text{off}}^{\text{T}_i,
\ell_i^{\romb}}$ to the local solution $u|_{\omega_i}$ is of great importance to the analysis of  multiscale
methods \cite{melenk1996partition,EFENDIEV2011937}. We  present relevant results in Section \ref{subsec:spectral} below.
\subsubsection*{Local harmonic extension bases}
This type of local multiscale bases is defined by local solvers over $\omega_i$. The number of
such local solvers is problem-dependent. It can be the space of all fine-scale finite element basis functions or the
solutions of some local problems with suitable choices of boundary conditions. In this work, we consider the following
$\kappa$-harmonic extensions to form the local multiscale space, which has been extensively used
in the literature. Specifically, given a fine-scale piecewise linear function $\delta_j^h(x)$ defined on
the boundary $\partial\omega_i$, let $\phi_{j}^{\hi}$ be the solution to the following
Dirichlet boundary value problem:
\begin{alignat}{2} \label{harmonic_ex}
-\nabla\cdot(\kappa(x) \nabla \phi_{j}^{\hi} ) &= 0
\quad &&\text{in} \, \, \, \omega_i,\\
\phi_{j}^{\hi}&=\delta_j^h &&\text{ on }\partial\omega_i,\nonumber
\end{alignat}
where $\delta_j^h(x):=\delta_{j,k}\,\text{ for all } j,k\in \textsl{J}_{h}(\omega_i)$ with $\delta_{j,k}$ denoting the Kronecker
delta symbol, and $\textsl{J}_{h}(\omega_i)$ denoting the set of all fine-grid boundary nodes on $\partial\omega_i$.
Let $L_i$ be the number of the local multiscale functions on $\omega_i$. Then the local multiscale space $V^{\hi}_{\text{snap}}$ on $\omega_i$ is defined by
\begin{align}\label{eq:Vharmonic}
V^{\hi}_{\text{snap}}:=\text{span}\{\phi_j^{\hi}: \quad 1\leq j\leq L_i\}.
\end{align}
Its approximation property will be discussed in Section \ref{subsec:harmonic}.

\subsection*{Discrete POD}
One challenge associated with the local multiscale space $V^{\hi}_{\text{snap}}$ lies in the fact that its
dimensionality can be very large, i.e., $L_i\gg1$, when the problem becomes increasingly complicated in the sense that there
are more multiple scales in the coefficient $\kappa$. Thus, the discrete POD is often employed on $\omega_i$ to reduce
the dimensionality of $V^{\hi}_{\text{snap}}$, while maintaining a certain accuracy.

The discrete POD proceeds as follows. {After obtaining} a large number of local multiscale  functions $\{\phi_{j}^{\hi}\}_{j=1}^{L_i}$, with $L_i\gg 1$, by solving the local problem \eqref{harmonic_ex}, we generate a {problem adapted subset of much smaller size} from these basis functions by means of singular
value decomposition, by taking only left singular vectors corresponding to the largest singular values. The resulting low-dimensional linear subspace with $\ell_i$ singular vectors is termed as the offline space of rank $\ell_i$.

The auxiliary spectral problem in the construction is to find $( \lambda_j^{\hi}, v_j)\in \mathbb{R}\times \mathbb{R}^{L_i}$ for $1\leq j\leq L_i$ with the eigenvalues $\{\lambda_j^{\hi}\}_{j=1}^{L_i}$ in a nondecreasing order (with multiplicity counted) such that
\begin{align} \label{offeig}
A^{\text{off}} v_j& = \lambda_j^{\hi} S^{\text{off}} v_j,\\
(S^{\text{off}} v_j,v_j)_{\ell^2}&=1\nonumber.
\end{align}
The matrices $A^{\text{off}}, S^{\text{off}}\in \mathbb{R}^{L_i\times L_i}$ are respectively defined by
\begin{equation*}
 \displaystyle A^{\text{off}} = [a_{mn}^{\text{off}}] = \int_{\omega_i} \kappa\nabla \phi_m^{\hi} \cdot \nabla \phi_n^{\hi}\dx \quad\text{ and }\quad
 \displaystyle S^{\text{off}} = [s_{mn}^{\text{off}}] = \int_{\omega_i}  \widetilde{\kappa} \phi_m^{\hi} \cdot\phi_n^{\hi}\dx .
\end{equation*}
Let $\mathbb{N}_{+}\ni \ell_i\leq L_i$ be a truncation number. Then we define the discrete POD-basis of rank $\ell_i$ by
\begin{align}\label{eq:pod-basis}
v_j^{\hi}:=\sum\limits_{k=1}^{L_i}(v_j)_{k}\phi_{k}^{\hi}\;\quad
\text{ for }j=1,\cdots,\ell_i,
\end{align}
with $(v_j)_{k}$ being the $k^{\text{th}}$ component of the eigenvector $v_j\in\mathbb{R}^{L_i}$. By the definition of the
discrete eigenvalue problem \eqref{offeig}, we have
\begin{align}\label{eq:podNorm}
(v_j^{\hi}, v_k^{\hi})_i =\delta_{jk} \quad \text{ and } \quad \int_{\omega_i}\kappa \nabla v_j^{\hi}\cdot\nabla v_k^{\hi}\dx =\lambda_j^{\hi}\delta_{jk} \qquad\text{ for all } 1\leq j,k\leq \ell_i.
\end{align}
The local offline space $V^{\text{H}_i,\ell_i}_{\text{off}}$ of rank $\ell_i$ is spanned by the first $\ell_i$
eigenvectors corresponding to the smallest eigenvalues for problem \eqref{offeig}:
\begin{align*}
V^{\text{H}_i,\ell_i}_{\text{off}} := \text{span}\left\{v_j^{\hi}: \quad 1\leq j\leq \ell_i  \right\}.
\end{align*}
Analogously, we can define a rank $\ell_i$ projection operator $\mathcal{P}^{\si,\ell_i}: V_{\text{snap}}^{\hi}\to
V_{\text{off}}^{\hi,\ell_i}$ for all $\mathbb{N}_{+}\ni \ell_i\leq L_i$ by
\begin{equation}\label{eqn:proj-pod}
\mathcal{P}^{\hi,\ell_i}v=\sum\limits_{j=1}^{\ell_i}(v,v_j^{\hi})_i v_j^{\hi}\ \ \text{ for all } v\in V_{\text{snap}}^{\hi}.
\end{equation}
This projection is crucial to derive the error estimate for the discrete POD basis.
Its approximation property will be discussed in Section \ref{sec:discretePOD}.

\subsection{Galerkin approximation}
\label{globcoupling}
Next we define three types of global multiscale basis functions based on the local multiscale basis functions introduced in
Section \ref{locbasis} by partition of unity functions subordinated to the set of coarse neighborhoods $\{\omega_i\}_{i=1}^N$.
This gives rise to three multiscale methods for solving Problem \eqref{eqn:pde} that can approximate reasonably the exact solution $u$ (or
the fine-scale solution $u_h$).

We begin with an initial coarse space $V^{\text{init}}_0 = \text{span}\{ \chi_i \}_{i=1}^{N}$.
The functions $\chi_i$ are the standard multiscale basis functions on each coarse element $K\in \mathcal{T}^{H}$ defined by
\begin{alignat}{2} \label{pou}
-\nabla\cdot(\kappa(x)\nabla\chi_i) &= 0  &&\quad\text{ in }\;\;K, \\
\chi_i &= g_i &&\quad\text{ on }\partial K, \nonumber
\end{alignat}
where $g_i$ is affine over $\partial K$ with $g_i(O_j)=\delta_{ij}$ for all $i,j=1,\cdots, N$. Recall that $\{O_j\}_{j=1}^{N}$ are the set of coarse nodes on $\mathcal{T}^{H}$.
\begin{remark}[Properties of $\chi_i$]\label{rem:chi}
The definition \eqref{pou} implies that $\text{supp}(\chi_i)=\omega_i$. Thus, we have
\begin{align}\label{eq:chi_supp}
\chi_i=0\quad \text{ on }\partial \omega_i.
\end{align}
Furthermore, the maximum principle implies
$0\leq \chi_i\leq 1.$
Note that under Assumption \ref{ass:coeff}, the gradient of the multiscale basis functions
$\{\chi_i\}$ are uniformly bounded \cite[Corollary 1.3]{li2000gradient}
\begin{align}\label{eq:gradientChi}
\|\nabla\chi_i\|_{L^{\infty}(\omega_i)}\leq C_0,
\end{align}
where the constant $C_0$ depends on $D$, the size and shape of $D_j$ for $j=1,\cdots,m$,
the space dimension $d$ and the coefficient $\kappa$, but it
is independent of the distances between the inclusions $D_k$ and $D_j$ for $k,j=1,\cdots, m$.
It is worth noting that the precise dependence of the constant $C_0$ on $\kappa$ is still
unknown. However, when the contrast $\Lambda=\infty$, it is
known that the constant $C_0$ will blow up as two inclusions approach each other, for
which the problem reduces to the perfect or insulated conductivity problem
\cite{bao2010gradient}. Such extreme cases are beyond the scope of the present work.
The constant $C_0$ also depends on coarse grid size $H$ with a possible scaling $H^{-1}$.
\end{remark}

Since the set of functions $\{\chi_i\}_{i=1}^{N}$ form partition of unity functions subordinated
to $\{\omega_i\}_{i=1}^{N}$, we can construct global
multiscale basis functions from the local multiscale basis functions discussed in Section \ref{locbasis}
\cite{melenk1996partition,EFENDIEV2011937}. Specifically, the global multiscale spaces $V_{\text{off}}^{\text{S}}$,
$V_{\text{snap}}$ and $V_{\text{off}} ^{\text{H}}$ are respectively defined by
\begin{equation}\label{eq:globalBasis}
\begin{aligned}
V_{\text{off}}^{\text{S}}  &:= \text{span} \{ \chi_i v_j^{\si},\chi_i v_{k}^{\ti},\chi_i v^{i}: \,  \, 1 \leq i \leq N,\,\,\, 1 \leq j \leq \ell_i^{\roma} \text{ and } 1 \leq k \leq \ell_i^{\romb} \text{ with }\ell_i^{\roma}+\ell_i^{\romb}=\ell_i-1\},
\\
V_{\text{snap}} &:= \text{span}\{ \chi_i\phi_{j}^{ \hi}:~~~ 1\leq i\leq N \text{ and }1\leq j \leq {L_i} \},\\
V_{\text{off}} ^{\text{H}} &:= \text{span} \{ \chi_i v_j^{\hi}: \,  \, 1 \leq i \leq N \, \, \,  \text{and} \, \, \, 1 \leq j \leq \ell_i  \}.
\end{aligned}
\end{equation}
Accordingly, the Galerkin approximations to Problem \eqref{eqn:pde} read respectively:
seeking $u_{\text{off}}^{\text{S}}\in V_{\text{off}}^{\text{S}}$, $u_{\text{snap}}\in V_{\text{snap}}$ and
$u_{\text{off}}^{\text{H}}\in V_{\text{off}}^{\text{H}}$, satisfying
\begin{align}
a(u_{\text{off}}^{\text{S}}, v) &= (f, v)_{D} \quad\text{for all} \,\,\, v \in V_{\text{off}}^{\text{S}},\label{cgvarform_spectral}\\
a(u_{\text{snap}}, v) &= (f, v)_{D} \quad\text{for all} \,\,\, v \in V_{\text{snap}},\label{cgvarform_snap}\\
a(u_{\text{off}}^{\text{H}}, v) &= (f, v)_{D}  \quad \text{for all} \,\,\, v \in V_{\text{off}}^{\text{H}}.\label{cgvarform_pod}
\end{align}
Note that, by its construction, we have the inclusion relation $V_{\text{off}}^{\text{H}}\subset V_{\text{snap}}$ for all
$1\leq \ell_i\leq L_i$ with $i=1,2,\cdots, N$. Hence, the Gakerkin orthogonality property \cite[Corollary 2.5.10]{MR2373954} implies
\[
\seminormE{u-u_{\text{off}}^{\text{H}}}{D}^2= \seminormE{u-u_{\text{snap}}}{D}^2+\seminormE{u_{\text{snap}}-u_{\text{off}}^{\text{H}}}{D}^2.
\]
Furthermore, we will prove in Section \ref{sec:discretePOD} that
$u_{\text{off}}^{\text{H}}\to u_{\text{snap}}$  in $ H^1_0(D),$
and the convergence rate is determined by $\max_{i=1,\cdots,N}\big\{(H^2\lambda_{\ell_i+1}^{\hi})^{-1/2}\big\}$.

The main goal of this work is to derive bounds on the errors $\seminormE{u-u_{\text{off}}^{\text{S}}}{D}$,
$\seminormE{u-u_{\text{snap}}}{D}$ and $\seminormE{u-u_{\text{off}}^{\text{H}}}{D}$. This
will be carried out in Section \ref{sec:error} below.

\section{Error estimates}\label{sec:error}
This section is devoted to the energy error estimates for the multiscale approximations.
The general strategy is as follows. First, we derive approximation properties to
the local solution $u|_{\omega_i}$, for the local multiscale spaces $V_{\text{off}}^{\text{S}_i,
\ell_i^{\roma}}$, $V_{\text{off}}^{\text{T}_i,\ell_i^{\romb}}$, $V_{\text{snap}}^{\text{H}_i}$
and $V_{\text{off}}^{\text{H}_i,\ell_i}$. Then we combine these local estimates together
with partition of unity functions to establish the desired global energy error estimates.

\subsection{Spectral bases approximate error}\label{subsec:spectral}
Note that the solution $u$ satisfies the following equation
\begin{equation*}
\left\{
\begin{aligned}
-\nabla\cdot(\kappa\nabla u)&=f \quad&&\text{ in } \omega_i,\\
-\kappa\frac{\partial u}{\partial n}&=-\kappa\frac{\partial u}{\partial n}\quad&&\text{ on }\partial \omega_i,
\end{aligned}
\right.
\end{equation*}
which can be split into three parts, namely
\begin{align}\label{eq:decomp}
u|_{\omega_i}=u^{i,\roma}+u^{i,\romb}+u^{i,\romc}.
\end{align}
Here, the three components $u^{i,\roma}$, $u^{i,\romb}$, and $u^{i,\romc}$ are respectively given by
\begin{equation}\label{eq:u-roma1}
\left\{
\begin{aligned}
-\nabla\cdot(\kappa\nabla u^{i,\roma})&=f-\bar{f}_i \quad&&\text{ in } \omega_i\\
-\kappa\frac{\partial u^{i,\roma}}{\partial n}&=0\quad&&\text{ on }\partial \omega_i,
\end{aligned}
\right.
\end{equation}
where $\bar{f}_i:=\int_{\omega_i}f\dx\times\frac{\widetilde{\kappa}}{\int_{\omega_i}\widetilde{\kappa}\dx}$,
\begin{equation*}
\left\{
\begin{aligned}
-\nabla\cdot(\kappa\nabla u^{i,\romb})&=0 \quad&&\text{ in } \omega_i\\
-\kappa\frac{\partial u^{i,\romb}}{\partial n}&=\kappa\frac{\partial u}{\partial n}-\dashint_{\partial\omega_i}\kappa\frac{\partial u}{\partial n}\quad&&\text{ on }\partial \omega_i,
\end{aligned}
\right.
\end{equation*}
and
\[
u^{i,\romc}=v^{i}\int_{\omega_i}f\dx
\]
with $v^i$ being defined in \eqref{eq:1-basis}. Clearly,
$u^{i,\romc}$ involves only one local solver.
We begin with an {\em a priori} estimate on $u^{i,\romb}$.
\begin{lemma}The following a priori estimate holds:
\begin{align}\label{eq:u2-bound}
\seminormE{u^{i,\romb}}{\omega_i}\leq \seminormE{u}{\omega_i}+H\Cpoin{\omega_i}^{1/2}\|f\|_{L^2_{\kappa^{-1}}(\omega_i)}.
\end{align}
\end{lemma}
\begin{proof}
Let $\widetilde{u}:=u^{i,\roma}+u^{i,\romc}$. Then it satisfies
\begin{equation*}\label{eq:u-roma}
\left\{
\begin{aligned}
-\nabla\cdot(\kappa\nabla \widetilde{u})&=f \quad&&\text{ in } \omega_i,\\
\kappa\frac{\partial \widetilde{u}}{\partial n}&=\frac{1}{{|\partial\omega_i|}}\int_{\omega_i}f\;\dx\quad&&\text{ on }\partial \omega_i.
\end{aligned}
\right.
\end{equation*}
To make the solution unique, we require $\int_{\partial\omega_i}\widetilde{u}\;{\rm d}s=0$.
Testing the first equation with $\widetilde{u}$ gives
\begin{align*}
\seminormE{\widetilde{u}}{\omega_i}^2=\int_{\omega_i}f\widetilde{u}\;\dx.
\end{align*}
Now Poincar\'{e} inequality \eqref{eq:poinConstant} and H\"{o}lder's inequality lead to
\begin{align*}
\seminormE{\widetilde{u}}{\omega_i}^2\leq \|f\|_{L^2_{\kappa^{-1}}(\omega_i)}\|\widetilde{u}\|_{L^2_{\kappa}(\omega_i)}
\leq H\Cpoin{\omega_i}^{1/2}\|f\|_{L^2_{\kappa^{-1}}(\omega_i)}\seminormE{\widetilde{u}}{\omega_i}.
\end{align*}
Therefore, we obtain
\begin{align*}
\seminormE{\widetilde{u}}{\omega_i}\leq H\Cpoin{\omega_i}^{1/2}\|f\|_{L^2_{\kappa^{-1}}(\omega_i)}.
\end{align*}
Finally, the desired result follows from the triangle inequality.
\end{proof}

Since $u^{i,\roma}\in L^{2}_{\widetilde{\kappa}}(\omega_i),$
$u^{i,\romb}\in L^{2}(\partial\omega_i)$,
and the series $\{v_j^{\si}\}_{j=1}^{\infty}\oplus \{1\}$ and $\{v_j^{\ti}\}_{j=1}^{\infty}$ form a complete orthogonal bases in $L^{2}_{\widetilde{\kappa}}(\omega_i)$ and $L^{2}(\partial\omega_i)$, respectively, $u^{i,\roma}$ and $u^{i,\romb}$ admit the following decompositions:
\begin{align}
u^{i,\roma}&=c_0(u^{i,\roma},1)_i+\sum\limits_{j=1}^{\infty}(u^{i,\roma},v_j^{\si})_i v_j^{\si},\label{eq:spectralU}\\
u^{i,\romb}&=\sum\limits_{j=1}^{\infty}(u^{i,\rm II},v_j^{\ti})_{\partial\omega_i} v_j^{\ti}.
\label{eq:spectralu2}
\end{align}
For any $n\in \mathbb{N}_{+}$, we employ the $n$-term truncation $u^{i,\roma}_n$ and $u^{i,\romb}_n$ to approximate $u^{i,\roma}$ and $u^{i,\romb}$, respectively,
on $\omega_i$:
\begin{align*}
u^{i,\roma}_n:=\mathcal{P}^{\si,n}u^{i,\roma}\in V_{\text{off}}^{\text{S}_i,n}
\quad \mbox{and}\quad
u^{i,\romb}_n:=\mathcal{P}^{\ti,n}u^{i,\romb}\in V_{\text{off}}^{\text{T}_i,n}.
\end{align*}
\begin{lemma}\label{lem:assF}
Assume that $f\in L^2_{\widetilde{\kappa}^{-1}}(D)$. Then there holds
\begin{align}\label{eq:f_norm}
\|f-\bar{f}_i\|_{L^{2}_{\widetilde{\kappa}^{-1}}(\omega_i)}^2
=\sum\limits_{j=1}^{\infty}\Big(\lambda_j^{\text{S}_i}\Big)^2 \Big|(u^{i,\roma},v_{j}^{\si})_i\Big|^2<\infty.
\end{align}
\end{lemma}
\begin{proof}
Since $f\in L^2_{\widetilde{\kappa}^{-1}}(D)$, by Lemma \ref{lem:L2Inv}, $f-\bar f_i$ admits the following spectral decomposition:
\begin{align}\label{eq:99}
f-\bar{f}_i=\Big(\int_{\omega_i}\widetilde{\kappa}\dx\Big)^{-1}
\Big(\int_{\omega_i}
(f-\bar{f}_i)\;\dx\Big)\widetilde{\kappa}
+\sum\limits_{j=1}^{\infty}
\Big(\int_{\omega_i}(f-\bar{f}_i)
v_j^{\si}\dx\Big)\widetilde{\kappa}v_j^{\si}.
\end{align}
By the definition of $\bar f_i$, the first term vanishes.
Thus, it suffices to compute the $j^{\text{th}}$ expansion coefficient
$\int_{\omega_i}(f-\bar{f}_i)v_j^{\si}\dx$ for $j=1,2,\cdots$, which follows from \eqref{eq:u-roma1}.
Indeed, testing \eqref{eq:u-roma1} with $v_j^{\si}$ yields
\begin{align*}
\int_{\omega_i}\Big(f-\bar{f}_i\Big)v_j^{\si}\dx
&=\int_{\omega_i}\kappa\nabla u^{i,\roma}\cdot\nabla v_j^{\si}\dx=\lambda_{j}^{\si}\int_{\omega_i}\widetilde{\kappa}u^{i,\roma} v_j^{\si}\dx
=\lambda_{j}^{\si}(u^{i,\roma}, v_j^{\si})_i.
\end{align*}
\end{proof}

Now we state an important approximation property of the operator $\mc{P}^{\si,\ell_i^{\roma}}$ of rank $\ell_i^{\roma}$ defined in \eqref{eq:FR_spec}.
\begin{proposition}\label{prop:projection}
Assume that $f\in L^2_{\widetilde{\kappa}^{-1}}(D)$ and $\ell_i^{\roma}\in \mathbb{N}_+$. Let
$u^{i,\roma}$ be the first component in \eqref{eq:decomp}. Then the projection
$\mc{P}^{\si,\ell_i^{\roma}}: L^2_{\widetilde{\kappa}}(\omega_i)\to V_{{\rm off}}^{\mathrm{S}_i,\ell_i^{\roma}}$ of rank
$\ell_i^{\roma}$ defined in \eqref{eq:FR_spec} has the following approximation properties:
\begin{align}
\normLT{u^{i,\roma}-\mc{P}^{\si,\ell_i^{\roma}}u^{i,\roma}}{\omega_i}&
\leq (\lambda_{\ell_i^{\roma}}^{\si})^{-1}\normLii{f}{\omega_i},\label{eq:3333}\\
\seminormE{u^{i,\roma}-\mc{P}^{\si,\ell_i^{\roma}}u^{i,\roma}}{\omega_i}&\leq  ({\lambda_{\ell_i^{\roma}}^{\si}})^{-\frac12}\normLii{f}{\omega_i}. \label{eq:4444}
\end{align}
\end{proposition}
\begin{proof}
The definitions \eqref{eq:spectralU} and \eqref{eq:FR_spec}, and the orthonormality of $\{v_j^{\si}\}_{j=1}^{\infty}\oplus\{1\}$
in $L^2_{\widetilde{\kappa}}(\omega_i)$ directly yield
\begin{align*}
\normLT{u^{i,\roma}-\mc{P}^{\si,\ell_i^{\roma}}u^{i,\roma}}{\omega_i}^2
&=\sum\limits_{j=\ell_i^{\roma}}^{\infty}(u^{i,\roma},v_j^{\si})_i^2
=\sum\limits_{j=\ell_i^{\roma}}^{\infty}(\lambda_j^{\si})^{-2}
(\lambda_j^{\si})^{2}(u^{i,\roma},v_j^{\si})_i^2\\
&\leq (\lambda_{\ell_i^{\roma}}^{\si})^{-2}\sum\limits_{j=\ell_i^{\roma}}^{\infty}
(\lambda_j^{\si})^{2}(u^{i,\roma},v_j^{\si})_i^2\\
&\leq (\lambda_{\ell_i^{\roma}}^{\si})^{-2}\normLii{f-\bar{f}_i}{\omega_i}^2,
\end{align*}
where in the last step we have used \eqref{eq:f_norm}.
Next, since the first term in the expansion
\eqref{eq:99} vanishes, we deduce that $f-\bar{f}_i$ is the $L^2_{\widetilde{\kappa}^{-1}}(\omega_i)$ projection
onto the codimension one subspace $L^2_{\widetilde{\kappa}^{-1}}(\omega_i)\backslash \{\widetilde{\kappa}\}$.
Thus,
\begin{align*}
\normLii{f-\bar{f}_i}{\omega_i}\leq \normLii{f}{\omega_i}.
\end{align*}
Plugging this inequality into the preceding estimate, we arrive at
\begin{align*}
\normLT{u^{i,\roma}-\mc{P}^{\si,\ell_i^{\roma}}u^{i,\roma}}{\omega_i}^2
\leq(\lambda_{\ell_i^{\roma}}^{\si})^{-2}\normLii{f}{\omega_i}^2,
\end{align*}
Taking the square root yields the first estimate. The second estimate can be derived in a similar manner.
\end{proof}

Next we give the approximation property of the finite rank operator $\mc{P}^{\ti,\ell_i^{\romb}}$ to
the second component of the local solution $u^{i,\romb}$, which relies on the regularity of
the very weak solution in the appendix.
\begin{lemma}\label{lemma:u2}
Let $\ell_i^{\roma}\in \mathbb{N}_+$ and let $u^{i,\romb}$ be the second component in \eqref{eq:decomp}. Then the projection $\mc{P}^{\ti,\ell_i^{\romb}}: L^2(\partial\omega_i)\to V_{\rm off}^{{\rm T}_i,\ell_i}$ of rank $\ell_i^{\romb}$ defined in \eqref{eq:steklov_spec} has the following approximation properties:
\begin{align}
\|{u^{i,\romb}-\mc{P}^{\ti,\ell_i^{\romb}}u^{i,\romb}}\|_{L^2(\partial\omega_i)}
&\leq  (\lambda_{\ell_i^{\romb}+1}^{\ti})^{-\frac12}
\Big(\seminormE{u}{\omega_i}+ H\sqrt{\Cpoin{\omega_i}}\|f\|_{L^2_{\kappa^{-1}}(\omega_i)}\Big),
\label{eq:5555}\\
\normLT{u^{i,\romb}-\mc{P}^{\ti,\ell_i^{\romb}}u^{i,\romb}}{\omega_i}
&\leq \Cw(\lambda_{\ell_i^{\romb}+1}^{\ti})^{-\frac12}
\Big(\seminormE{u}{\omega_i}+
H\sqrt{\Cpoin{\omega_i}}\|f\|_{L^2_{\kappa^{-1}}(\omega_i)}\Big), \label{eq:56789}\\
\int_{\omega_i}\chi_i^2\kappa |\nabla (u^{i,\romb}-\mc{P}^{\ti,\ell_i^{\romb}}u^{i,\romb}) |^2\dx
&\leq 8H^{-2}\Cw^2(\lambda_{\ell_i^{\romb}+1}^{\ti})^{-1}
\Big(\seminormE{u}{\omega_i}^2+ H^2\Cpoin{\omega_i}\|f\|_{L^2_{\kappa^{-1}}(\omega_i)}^2\Big).
\label{eq:777}
\end{align}
\end{lemma}
\begin{proof}
The inequality \eqref{eq:5555} follows from the expansion \eqref{eq:spectralu2},  \eqref{eq:steklov_spec}
and \eqref{eq:u2-bound}, and the fact that $u^{i,\romb}\in H^1_{\kappa}(\omega_i)$.
Indeed, we obtain from \eqref{eq:spectralu2} and the orthonomality of $\{v_j^{\ti}\}_{j=1}^{\infty}$ in $L^2(\partial\omega_i)$ that
\begin{align*}
\|{u^{i,\romb}-\mc{P}^{\ti,\ell_i^{\romb}}u^{i,\romb}}\|_{L^2(\partial\omega_i)}^2
&=\sum\limits_{j>\ell_i^{\romb}}|(u^{i,\romb},v_j^{\ti})_{\partial\omega_i}|^2
=\sum\limits_{j>\ell_i^{\romb}}(\lambda_j^{\ti})^{-1}\lambda_j^{\ti}|(u^{i,\romb},v_j^{\ti})_{\partial\omega_i}|^2\\
&\leq (\lambda_{\ell_i^{\romb}+1}^{\ti})^{-1}\sum\limits_{j>\ell_i^{\romb}}\lambda_j^{\ti}|(u^{i,\romb},v_j^{\ti})_{\partial\omega_i}|^2.
\end{align*}
Then the estimate \eqref{eq:5555} follows from \eqref{eq:u2-bound} and the identity
$
\langle u^{i,\romb},u^{i,\romb} \rangle_i=\sum_{j=1}^{\infty}\lambda_j^{\ti}|(u^{i,\romb},v_j^{\ti})_{\partial\omega_i}|^2.
$
To prove \eqref{eq:56789}, we first write the local error equation for
$e:=u^{i,\romb}-\mc{P}^{\ti,\ell_i^{\romb}}u^{i,\romb}$ by
\begin{equation}\label{eq:222}
\left\{
\begin{aligned}
-\nabla\cdot(\kappa\nabla e)&=0 \quad&&\text{ in } \omega_i,\\
e&=u^{i,\romb}-\mc{P}^{\ti,\ell_i^{\romb}}u^{i,\romb}\quad&&\text{ on }\partial \omega_i.
\end{aligned}
\right.
\end{equation}
Now Theorem \ref{lem:very-weak} yields
\begin{align*}
\normLT{u^{i,\romb}-\mc{P}^{\ti,\ell_i^{\romb}}u^{i,\romb}}{\omega_i}\leq \text{C}_{\text{weak}}\|{u^{i,\romb}-\mc{P}^{\ti,\ell_i^{\romb}}u^{i,\romb}}\|_{L^2(\partial\omega_i)}
\end{align*}
for some constant $\text{C}_{\text{weak}}$ independent of the coefficient $\kappa$. This, together with \eqref{eq:5555}, proves \eqref{eq:56789}.

To derive the energy error estimate from the $L^2_{\widetilde{\kappa}}(\omega_i)$ error estimate, we employ a Cacciopoli type inequality.
Note that $\chi_i=0$ on the boundary $\partial \omega_i$, cf \eqref{eq:chi_supp}.
Multiplying the first equation in \eqref{eq:222} with $\chi_i^2 e_n$ and then integrating over $\omega_i$ and integration by parts lead to
\begin{align*}
\int_{\omega_i}\chi_i^2\kappa|\nabla e_n|^2\dx=-2\int_{\omega_i}\kappa \nabla e_n\cdot \nabla \chi_i\chi_i e_n\;\dx.
\end{align*}
Together with H\"{o}lder's inequality and Young's inequality, we arrive at
\begin{align*}
\int_{\omega_i}\chi_i^2\kappa|\nabla e_n|^2\dx\leq 4\int_{\omega_i}\kappa|\nabla\chi_i|^2 e_n^2\,\dx.
\end{align*}
Further, the definition of $\widetilde{\kappa}$ in \eqref{defn:tildeKappa} yields
\begin{align*}
\int_{\omega_i}\chi_i^2\kappa|\nabla e_n|^2\dx\leq 4H^{-2}\int_{\omega_i}\widetilde{\kappa} e_n^2\,\dx.
\end{align*}
Now  \eqref{eq:56789} and Young's inequality yield \eqref{eq:777}.
This completes the proof of the lemma.
\end{proof}
\begin{remark}
It is worth emphasizing that the local energy estimates \eqref{eq:4444} and \eqref{eq:777} are derived under almost no restrictive
assumptions besides the mild condition $f\in L^2_{\widetilde{\kappa}^{-1}}(D)$. This estimate is new to the best of
our knowledge. The authors \cite{EFENDIEV2011937} utilized the Cacciopoli inequality to derive similar
estimates, which, however, incurs some (implicit) assumptions on the problem. Hence, the estimates \eqref{eq:4444} and \eqref{eq:777}
are important for justifying the local spectral approach.
\end{remark}
Finally, we present the rank-$\ell_i$ approximation to $u|_{\omega_i}$, where $\ell_i:=\ell_i^{\roma}+\ell_i^{\romb}+1$ with $\ell_i^{\roma}, \ell_i^{\romb}\in \mathbb{N}$ for all $i=1,2,\cdots, N$:
\begin{align}\label{eq:spectral-finiteRank}
\widetilde{u}_i:=\mc{P}^{\si,\ell_i^{\roma}}u^{i,\roma}_i+\mc{P}^{\ti,\ell_i^{\romb}}u^{i,\romb}_i+u^{i,\romc}.
\end{align}

Now, we  present an error estimate for the Galerkin approximation $u_{\text{off}}^{\text{S}}$ based on
the local spectral basis, cf. \eqref{cgvarform_spectral}. Our proof is inspired by the partition of unity finite
element method (FEM) \cite[Theorem 2.1]{melenk1996partition}.
\begin{lemma}\label{lem:spectralApprox}
Assume that $f\in L^2_{\widetilde{\kappa}^{-1}}(D)\cap L^2_{{\kappa}^{-1}}(D)$ and $\ell_i^{\roma}, \ell_i^{\romb}\in \mathbb{N}$ for all $i=1,2,\cdots, N$. Let $u$ be the solution to Problem \eqref{eqn:pde}. Denote $V_{{\rm off}}^{{\rm S}}\ni w_{{\rm off}}^{{\rm S}}:=\sum\limits_{i=1}^{N}\chi_i \widetilde{u}_i$. Then there holds
\begin{equation*}
\begin{aligned}
\seminormE{u-w_{{\rm off}}^{{\rm S}}}{D}
&\leq {2}C_{{\rm ov}}\max_{i=1,\cdots,N}\big\{({H\lambda_{\ell_i^{\roma}}^{\si}})^{-1}+
({\lambda_{\ell_i^{\roma}}^{\si}})^{-\frac12}
\big\}\normLii{f}{D}\\
&+7C_{{\rm ov}}\Cw{\rm C}_{{\rm poin}}\max_{i=1,\cdots,N}
\big\{(H^2\lambda_{\ell_i^{\romb}+1}^{\ti})^{-\frac12}\big\}
\|f\|_{L^2_{\kappa^{-1}}(D)},
\end{aligned}
\end{equation*}
where ${\rm C}_{\text{poin}}:={\rm diam}(D)\Cpoin{D}^{1/2}+H\max_{i=1,\cdots,N}\{\Cpoin{\omega_i}^{1/2}\}$.
\end{lemma}
\begin{proof}
Let $\locv{e}{}{}:=u-w_{\text{off}}^{\text{S}}$.
Then the property of the partition of unity of $\{\chi_i\}_{i=1}^{N}$ leads to
\[
\locv{e}{}{}=\sum\limits_{i=1}^{N}\chi_i\locv{e}{i}{} \qquad\text{ with }
\qquad\locv{e}{i}{}:=(u^{\roma}_i-\mc{P}^{\si,\ell_i^{\roma}}u^{\roma}_i)+(u^{i,\romb}_i-\mc{P}^{\ti,\ell_i^{\romb}}u^{i,\romb}_i)
:=\locv{e}{i}{\roma}+\locv{e}{i}{\romb}.
\]
Taking its squared energy norm and using the overlap condition \eqref{eq:overlap}, we arrive at
\begin{align*}
\int_{D}\kappa|\nabla \locv{e}{}{}|^2\dx&=\int_{D}\kappa|\sum\limits_{i=1}^{N}
\nabla(\chi_i\locv{e}{i}{})|^2\dx
\leq\Cov\sum\limits_{i=1}^{N}\int_{\omega_i}\kappa
|\nabla(\chi_i\locv{e}{i}{})|^2\dx.
\end{align*}
This and Young's inequality together imply
\begin{align}\label{eq:1111}
\int_{D}\kappa|\nabla \locv{e}{}{}|^2\dx
\leq 2\Cov\sum\limits_{i=1}^{N}
\Big(\int_{\omega_i}\kappa
|\nabla(\chi_i\locv{e}{i}{\roma})|^2\dx+\int_{\omega_i}\kappa
|\nabla(\chi_i\locv{e}{i}{\romb})|^2\dx\Big).
\end{align}
It remains to estimate the two integral terms in the bracket. By Cauchy-Schwarz inequality and
the definition \eqref{defn:tildeKappa} of $\tilde \kappa$, we obtain
\begin{align}
\int_{\omega_i}\kappa
|\nabla(\chi_i\locv{e}{i}{\roma})|^2\dx&\leq 2\Big( \int_{\omega_i}\big(\kappa\sum\limits_{j=1}^{N}
|\nabla\chi_j|^2\big)|\locv{e}{i}{\roma}|^2\dx+\int_{\omega_i}\kappa\chi_i^2
|\nabla\locv{e}{i}{\roma}|^2\dx\Big)\nonumber\\
&\leq 2\Big( H^{-2}\int_{\omega_i}\widetilde{\kappa}|\locv{e}{i}{\roma}|^2\dx+\int_{\omega_i}\chi_i^2\kappa
|\nabla\locv{e}{i}{\roma}|^2\dx\Big).\label{eq:444}
\end{align}
Then Proposition \ref{prop:projection} yields
\begin{align*}
\int_{\omega_i}\kappa
|\nabla(\chi_i\locv{e}{i}{\roma})|^2\dx
\leq 2\Big((H\lambda_{\ell_i^{\roma}}^{\si})^{-2}
+(\lambda_{\ell_i^{\roma}}^{\si})^{-1}\Big)\normLii{f}{\omega_i}^2.
\end{align*}
Analogously, we can derive the following upper bound for the second term:
\begin{align*}
\int_{\omega_i}\kappa
|\nabla(\chi_i\locv{e}{i}{\romb})|^2\dx
\leq 20\Cw^2(H^2\lambda_{\ell_i^{\romb}+1}^{\ti})^{-1}
\Big(\seminormE{u}{\omega_i}^2+H^2\Cpoin{\omega_i}\|f\|_{L^2_{\kappa^{-1}}(\omega_i)}^2\Big).
\end{align*}
Inserting these two estimate into \eqref{eq:1111} gives
\begin{align*}
\int_{D}\kappa|\nabla \locv{e}{}{}|^2\dx&\leq 4C_{\text{ov}}\sum\limits_{i=1}^{N}\Big((H\lambda_{\ell_i}^{\si})^{-2}
+(\lambda_{\ell_i^{\roma}}^{\si})^{-1}\Big)
\normLii{f}{\omega_i}^2\\
&+40C_{\text{ov}}\sum\limits_{i=1}^{N}\Cw^2(H^2\lambda_{\ell^{\romb}_{i}+1}^{\ti})^{-1}
\Big(\seminormE{u}{\omega_i}^2+\Cpoin{\omega_i}H^2 \|f\|_{L^2_{\kappa^{-1}}(\omega_i)}^2\Big).
\end{align*}
Finally, the overlap condition \eqref{eq:overlap} leads to
\begin{equation}\label{eq:999}
\begin{aligned}
\int_{D}\kappa|\nabla \locv{e}{}{}|^2\dx&\leq 4C_{\text{ov}}^2\max_{i=1,\cdots,N}\Big\{(H\lambda_{\ell_i^{\roma}}^{\si})^{-2}
+(\lambda_{\ell_i^{\roma}}^{\si})^{-1}\Big\}
\normLii{f}{D}^2\\
&+40C_{\text{ov}}^2\Cw^2\max_{i=1,\cdots,N}\{(H^2
{\lambda_{\ell_i^{\romb}+1}^{\ti}})^{-1}
\}\\
&\times\Big(\seminormE{u}{D}^2+ H^2\max_{i=1,\cdots,N}\{{\Cpoin{\omega_i}}
\}\|f\|_{L^2_{\kappa^{-1}}(D)}^2\Big).
\end{aligned}
\end{equation}
Furthermore, since $f\in L^2_{\kappa^{-1}}(D)$, we obtain from Poincar\'{e}'s inequality \eqref{eq:poinConstantG} the {\em a priori} estimate
\begin{align}\label{eq:888}
\seminormE{u}{D}\leq \text{diam}(D)\Cpoin{D}^{1/2}\normLi{f}{D}.
\end{align}
Indeed, we can get by \eqref{eq:poinConstantG} that
\begin{align*}
\int_D \kappa u^{2}\dx\leq \text{diam}(D)^2{\Cpoin{D}}\int_{D}\kappa|\nabla u|^2\dx.
\end{align*}
Testing \eqref{eqn:pde} with $u\in V$, by H\"{o}lder's inequality, leads to
\begin{align*}
\int_{D}\kappa|\nabla u|^2\dx&=\int_D fu\; \dx
\leq \|f\|_{L^2_{\kappa^{-1}}(D)}\|u\|_{L^2_{\kappa}(D)}.
\end{align*}
These two inequalities together imply \eqref{eq:888}. Inserting \eqref{eq:888} into \eqref{eq:999} shows the desired assertion.
\end{proof}

An immediate corollary of Lemma \ref{lem:spectralApprox}, after appealing to the Galerkin orthogonality
property \cite[Corollary 2.5.10]{MR2373954}, is the following energy error between $u$ and $u_{\text{off}}^{\text{S}}$:
\begin{proposition}\label{prop:Finalspectral}
Assume that $f\in L^2_{\widetilde{\kappa}^{-1}}(D)\cap L^2_{{\kappa}^{-1}}(D)$ and let $\ell_i^{\roma}, \ell_i^{\romb}\in \mathbb{N}_{+}$ for all $i=1,2,\cdots, N$. Let $u\in V$ and $u_{{\rm off}}^{{\rm S}}\in V_{{\rm off}}^{{\rm S}}$ be the solutions to Problems \eqref{eqn:pde} and \eqref{cgvarform_spectral}, respectively. There holds
\begin{equation}\label{eq:snapErr}
\begin{aligned}
\seminormE{u-u_{{\rm off}}^{{\rm S}}}{D}&:=\min\limits_{w\in  V_{{\rm off}}^{{\rm S}}}\seminormE{u-w}{D}\\
&\leq \sqrt{2}C_{{\rm ov}}\max_{i=1,\cdots,N}\big\{({H\lambda_{\ell_i^{\roma}}^{\si}})^{-1}+
({\lambda_{\ell_i^{\roma}}^{\si}})^{-\frac12}
\big\}\normLii{f}{D}\\
&+7C_{{\rm ov}}\Cw{\rm C}_{{\rm poin}}\max_{i=1,\cdots,N}
\big\{({H^2\lambda_{\ell_i^{\romb}+1}^{\ti}})^{-\frac12}\big\}
\|f\|_{L^2_{\kappa^{-1}}(D)}.
\end{aligned}
\end{equation}
\end{proposition}

\begin{remark}\label{rem:spectral}
According to Proposition \ref{prop:Finalspectral}, the convergence rate is essentially determined by two factors:
the smallest eigenvalue $\lambda_{\ell_i}^{\si}$ that is not included in the local spectral basis and
the coarse mesh size $H$. A proper balance between them is necessary for the convergence. For any
fixed $H>0$, in view of the eigenvalue problems \eqref{eq:spectral} and \eqref{eq:steklov},
a simple scaling argument implies
\begin{align*}
H^2\lambda^{\si}_{\ell_i^{\roma}}\to \infty \quad\text{ and }\quad H\lambda^{\ti}_{\ell_i^{\romb}}\to \infty,\quad \text{ as }\quad \ell_i^{\roma},\ell_i^{\romb} \to \infty.
\end{align*}
Hence, assuming that $\ell_i^{\roma}$ and $\ell_i^{\romb}$ are sufficiently large such that
$H^2\lambda^{\si}_{\ell_i^{\romb}}\geq 1$ and $H\lambda^{\ti}_{\ell_i^{\romb}}\geq H^{-3}$,
from Proposition \ref{prop:Finalspectral}, we obtain
\begin{align}\label{eq:aaaa}
\seminormE{u-u_{{\rm off}}^{{\rm S}}}{D}\lesssim H \Big(\normLii{f}{D}+\normLi{f}{D}\Big).
\end{align}
Note that the estimate of type \eqref{eq:aaaa} is the main goal of the convergence analysis for many multiscale methods
\cite{MR1660141, MR2721592, li2017error}. In practice, the numbers $\ell_i^{\roma}$ and $\ell_i^{\romb}$ of local multiscale functions
fully determine the computational complexity of the multiscale solver for Problem \eqref{cgvarform_spectral}
at the offline stage. However, its optimal choice rests on the decay rate of the nonincreasing sequences
$\big\{(\lambda_{n}^{\si})^{-1}\big\}_{n=1}^{\infty}$ and $\big\{(\lambda_{n}^{\ti})^{-1}\big\}_{n=1}^{\infty}$. The precise characterization of eigenvalue decay estimates
for heterogeneous problems seems poorly understood at present, and the topic is beyond the scope of the present work.
\end{remark}

\subsection{Harmonic extension bases approximation error}\label{subsec:harmonic}
By the definition of the local harmonic extension snapshot space $\locV{i}{snap}$ in \eqref{harmonic_ex}
and \eqref{eq:Vharmonic}, there exists $u^{i}_{\text{snap}}\in \locV{i}{snap}$ satisfying
\begin{align}\label{eq:usnap}
u^{i}_{\text{snap}}:=u_{h} \qquad\text{ on }\quad \partial \omega_i.
\end{align}

In the error analysis below, the weighted Friedrichs
(or Poincar\'{e}) inequalities play an important role. These inequalities require certain conditions
on the coefficient $\kappa$ and domain $D$ that in general are not fully understood. Assumption
\ref{ass:coeff} is one sufficient condition for the weighted Friedrichs inequality \cite{galvis2010domain,pechstein2012weighted}.

Now we can derive the following local energy error estimate.
\begin{lemma}\label{lem:energyHA}
Let $\locv{e}{i}{snap}=u_h-u^{i}_{\text{snap}}$. Then there holds
\begin{align}\label{eq:energyHA}
\seminormE{\locv{e}{i}{snap}}{\omega_i}\leq {H}\sqrt{\Cpoin{\omega_i}}\normLi{f}{\omega_i}.
\end{align}
\end{lemma}
\begin{proof}
Indeed, by definition, the following error equation holds:
\begin{equation}\label{eq:locErr}
\left\{
\begin{aligned}
-\nabla\cdot(\kappa\nabla \locv{e}{i}{snap})&=f \quad&&\text{ in } \omega_i,\\
\locv{e}{i}{snap}&=0 \quad&&\text{ on }\partial \omega_i.
\end{aligned}\right.
\end{equation}
Then \eqref{eq:poinConstant} and H\"{o}lder's inequality give the assertion.
\end{proof}

\begin{lemma}
Assume that $f\in L^2_{{\kappa}^{-1}}(D)$ and $\ell_i\in \mathbb{N}_{+}$ for all $i=1,2,\cdots, N$. Let $u_h\in V_h$ be the unique solution to Problem \eqref{eqn:weakform_h}. Denote $V_{{\rm snap}}\ni w_{{\rm snap}}:=\sum_{i=1}^{N}\chi_i u^{i}_{{\rm snap}}$. Then there holds
\begin{align*}
\seminormE{u_h-w_{{\rm snap}}}{D}\leq \sqrt{2C_{{\rm ov}}}H\max_{i=1,\cdots,N}\Big\{ C_0H\Cpoin{\omega_i}+\sqrt{\Cpoin{\omega_i}}\Big\}\normLi{f}{D}.
\end{align*}
\end{lemma}
\begin{proof}
Let $\locv{e}{}{snap}:=u_h-w_{\text{snap}}$. Since $\{\chi_i\}_{i=1}^{N}$ forms a set of partition
of unity functions subordinated to the set $\{\omega_i\}_{i=1}^{N}$, we deduce
\[
\locv{e}{}{snap}=\sum\limits_{i=1}^{N}\chi_i\locv{e}{i}{snap},
\]
where $\locv{e}{i}{snap}:=u_h-u^{i}_{\text{snap}}$ is the local error on $\omega_i$.
Taking its squared energy norm and using the overlap condition \eqref{eq:overlap}, we arrive at
\begin{align}\label{eq:0000}
\int_{D}\kappa|\nabla \locv{e}{}{snap}|^2\dx&=\int_{D}\kappa|\sum\limits_{i=1}^{N}
\nabla(\chi_i\locv{e}{i}{snap})|^2\dx
\leq\Cov\sum\limits_{i=1}^{N}\int_{\omega_i}\kappa
|\nabla(\chi_i\locv{e}{i}{snap})|^2{\color{blue} \dx}.
\end{align}
It remains to estimate the integral term. Young's inequality gives
\begin{align*}
\int_{\omega_i}\kappa
|\nabla(\chi_i\locv{e}{i}{snap})|^2\dx&\leq 2\Big( \int_{\omega_i}\big(\kappa
|\nabla\chi_i|^2\big)|\locv{e}{i}{snap}|^2\dx+\int_{\omega_i}\kappa
|\nabla\locv{e}{i}{snap}|^2\dx\Big).
\end{align*}
Taking \eqref{eq:gradientChi} and \eqref{eq:poinConstant} into account, we get
\begin{align*}
\sum\limits_{i=1}^{N}\int_{\omega_i}\kappa
|\nabla(\chi_i\locv{e}{i}{snap})|^2\dx&\leq 2\sum\limits_{i=1}^{N}\Big( C_0^2H^2\Cpoin{\omega_i}+1\Big)\int_{\omega_i}\kappa
|\nabla\locv{e}{i}{snap}|^2\dx.
\end{align*}
This and \eqref{eq:energyHA} yield
\begin{align*}
\sum\limits_{i=1}^{N}\int_{\omega_i}\kappa
|\nabla(\chi_i\locv{e}{i}{snap})|^2\dx
\leq 2\sum\limits_{i=1}^{N}\Big(C_0^2H^2\Cpoin{\omega_i}+1\Big)\times H^2\Cpoin{\omega_i}\normLi{f}{\omega_i}^2.
\end{align*}
Finally, the overlap condition \eqref{eq:overlap} and inequality \eqref{eq:0000} show the desired assertion.
\end{proof}

Finally, we derive an energy error estimate for the conforming Galerkin approximation to Problem \eqref{eqn:pde}
based on the multiscale space $V_{\text{snap}}$.
\begin{proposition}\label{prop:FinalSnap}
Assume that $f\in L^2_{{\kappa}^{-1}}(D)$. Let $u\in V$ and $u_{{\rm snap}}\in V_{{\rm snap}}$ be the solutions to Problems \eqref{eqn:pde} and \eqref{cgvarform_snap}, respectively. Then there holds
\begin{align*}
\seminormE{u-u_{{\rm snap}}}{D}\leq \sqrt{2C_{{\rm ov}}}H\max_{i=1,\cdots,N}\Big\{ C_0H\Cpoin{\omega_i}&+\sqrt{\Cpoin{\omega_i}}\Big\}\normLi{f}{D}+\min\limits_{v_h\in V_h}\seminormE{u-v_h}{D}.
\end{align*}
\end{proposition}
\begin{proof}
This assertion follows directly from the Galerkin orthogonality property \cite[Corollary 2.5.10]{MR2373954},
the triangle inequality and the fine-scale {\em a priori} estimate \eqref{eq:fineApriori}.
\end{proof}
\subsection{Discrete POD approximation error}\label{sec:discretePOD}
Now we turn to the discrete POD approximation.
First, we present an {\em a priori} estimate for Problem \eqref{eqn:weakform_h}.
It will be used to derive the energy estimate for $u_{\text{snap}}^i$ defined in \eqref{eq:usnap}.
\begin{lemma}
Assume that $f\in L^2_{{\kappa}^{-1}}(D)$. Let $u_h\in V_h$ be the solution to Problem \eqref{eqn:weakform_h}. Then there holds
\begin{align}
\seminormE{u_h}{D}&\leq 2{\rm diam}(D)\sqrt{\Cpoin{D}}\normLi{f}{D}. \label{eq:uApriori}
\end{align}
\end{lemma}
\begin{proof}
In analogy to \eqref{eq:888}, we obtain
\begin{align*}
\seminormE{u}{D}&\leq \text{diam}(D)\sqrt{\Cpoin{D}}\normLi{f}{D},\\
\seminormE{u_h}{D}&\leq \text{diam}(D)\sqrt{\Cpoin{D}}\normLi{f}{D}.
\end{align*}
This and the triangle inequality lead to the desired assertion.
\end{proof}

Let $u^{i}_{\text{snap}}\in \locV{i}{snap}$ be defined in \eqref{eq:usnap}.
Then we deduce from \eqref{eq:energyHA} and the triangle inequality that
\begin{align}\label{usnap:apriori}
\seminormE{u_{\text{snap}}^i}{\omega_i}\leq  {H}\Cpoin{\omega_i}^{1/2}\normLi{f}{\omega_i}+\seminormE{u_h}{\omega_i}.
\end{align}
Note that the series $\{v_j^{\hi}\}_{j=1}^{L_i}$ forms a set of orthogonal basis in $ \locV{i}{snap}$, cf.
\eqref{eq:podNorm}. Therefore, the function $u^{i}_{\text{snap}}\in \locV{i}{snap}$ admits the following expansion
\begin{align}\label{eq:usnap_expand}
u^{i}_{\text{snap}}=\sum\limits_{j=1}^{L_i}(u^{i}_{\text{snap}},v_j^{\hi})_i v_j^{\hi}.
\end{align}
To approximate $u^{i}_{\text{snap}}$ in the space $V_{\text{off}}^{\hi,n}$ of dimension $n$ for some $\mathbb{N}_{+}\ni n\leq L_i$,
we take its first $n$-term truncation:
\begin{align}\label{eq_uin}
u^i_n:=\mathcal{P}^{\hi,n}u^{i}_{\text{snap}}=\sum\limits_{j=1}^{n}(u^{i}_{\text{snap}},v_j^{\hi})_i v_j^{\hi},
\end{align}
where the projection operator $\mathcal{P}^{\hi,n}$ is defined in \eqref{eqn:proj-pod}.

The next result provides the approximation property of $u^i_n$ to $u^{i}_{\text{snap}}$ in the $L^2_{\widetilde{\kappa}}(\omega_i)$ norm:
\begin{lemma}\label{lem:5.1}
Assume that $f\in L^2_{{\kappa}^{-1}}(D)$. Let $u^{i}_{{\rm snap}}\in \locV{i}{\rm snap}$ and $u^i_n\in V_{{\rm off}}^{\hi,n}$ be defined in \eqref{eq:usnap} and \eqref{eq_uin} for $\mathbb{N}_{+}\ni n\leq L_i$, respectively.  Then there holds
\begin{align*}
\| u^{i}_{{\rm snap}}-u^i_n \|_{L^2_{\widetilde{\kappa}}(\omega_i)}\leq \sqrt{2}(\lambda_{n+1}^{\hi})^{-1/2}\Big(   {H}\sqrt{\Cpoin{\omega_i}}\normLi{f}{\omega_i}+\seminormE{u_h}{\omega_i}\Big).
\end{align*}
\end{lemma}
\begin{proof}
It follows from the expansion \eqref{eq:usnap_expand} and \eqref{eq:podNorm} that
\begin{align*}
\int_{\omega_i }\kappa|\nabla u^{i}_{\text{snap}}|^2\dx=\sum\limits_{j=1}^{L_i}|(u^{i}_{\text{snap}},v_j^{\hi})_i|^2 \lambda_j^{\hi}.
\end{align*}
Together with \eqref{usnap:apriori}, we get
\begin{align}\label{eq:333}
\sum\limits_{j=1}^{L_i}|(u^i_{\rm snap},v_j^{\hi})_i|^2 \lambda_j^{\hi}\leq 2\Big(  {H}^2{\Cpoin{\omega_i}}\normLii{f}{\omega_i}^2+\seminormE{u_h}{\omega_i}^2\Big).
\end{align}
Meanwhile, the combination of \eqref{eq_uin}, \eqref{eq:usnap_expand} and \eqref{eq:podNorm} leads to
\begin{align*}
\| u^{i}_{\text{snap}}-u^i_n \|_{L^2_{\widetilde{\kappa}}(\omega_i)}^2&=\sum\limits_{j=n+1}^{L_i}|(u^{i}_{\text{snap}},v_j^{\hi})_i|^2
=\sum\limits_{j=n+1}^{L_i}(\lambda_j^{\hi})^{-1}\lambda_j^{\hi}|(u^{i}_{\text{snap}},v_j^{\hi})_i|^2\\
&\leq (\lambda_{n+1}^{\hi})^{-1}\sum\limits_{j=n+1}^{L_i}\lambda_j^{\hi}|(u^{i}_{\text{snap}},v_j^{\hi})_i|^2.
\end{align*}
Further, an application of \eqref{eq:333} implies
\begin{align*}
\| u^{i}_{\text{snap}}-u^i_n \|_{L^2_{\widetilde{\kappa}}(\omega_i)}^2
\leq (\lambda_{n+1}^{\hi})^{-1}\times 2\Big(  {H}^2{\Cpoin{\omega_i}}\normLi{f}{\omega_i}^2+\seminormE{u_h}{\omega_i}^2\Big).
\end{align*}
Finally, taking the square root on both sides shows the desired result.
\end{proof}

Note that for all $\mathbb{N}_{+}\ni n\leq L_i$, both approximations $u^{i}_{\text{snap}}$ and $u^i_n$ are $\kappa$-harmonic functions.
Thus, we can apply the argument in the proof of \eqref{eq:777} to get the following local energy error estimate.
\begin{lemma}\label{lem:5.2}
Let $u^{i}_{{\rm snap}}\in \locV{i}{\rm snap}$ and $u^i_n\in V_{{\rm off}}^{\hi,n}$ be defined in \eqref{eq:usnap}
and \eqref{eq_uin} for all $\mathbb{N}_{+}\ni n\leq L_i$.  Then there holds
\begin{align*}
\int_{\omega_i}\chi_i^2\kappa |\nabla (u^{i}_{{\rm snap}}-u^i_n) |^2\dx
&\leq 4H^{-2}\int_{\omega_i}\widetilde{\kappa} (u^{i}_{{\rm snap}}-u^i_n)^2\,\dx.
\end{align*}
\end{lemma}
\begin{proof}
The proof is analogous to that for \eqref{eq:777}, and thus omitted.
\end{proof}

With the help of local estimates presented in Lemmas \ref{lem:5.1} and \ref{lem:5.2},
we can now  bound the energy error for the POD method by means of the partition of unity
FEM \cite[Theorem 2.1]{melenk1996partition}.
\begin{lemma}\label{lem:4.6}
Assume that $f\in L^2_{\kappa^{-1}}(D)$.
For all $\mathbb{N}_{+}\ni \ell_i\leq L_i$, denote $V_{{\rm snap}}\ni w_{{\rm snap}}:
=\sum_{i=1}^{N}\chi_i u^{i}_{{\rm snap}}$ and $V_{{\rm off}}^{\text{H}}\ni w_{{\rm off}}^{{\rm H}}
:=\sum_{i=1}^{N}\chi_i u^{i}_{\ell_i}$. Then there holds
\begin{align*}
\seminormE{w_{{\rm snap}}-w_{{\rm off}}^{{\rm H}}}{D}\leq \sqrt{20C_{{\rm ov}}}&\max_{i=1,\cdots,N}
\Big\{{(H^{2}\lambda_{\ell_i+1}^{\hi})^{-1/2}}\Big\}
C_1\normLi{f}{D},
\end{align*}
where the constant $C_1$ is given by
$C_1:=H\max_{i=1,\cdots,N}\big\{\sqrt{\Cpoin{\omega_i}}\big\}+2{\rm diam}(D)\sqrt{\Cpoin{D}}.$
\end{lemma}
\begin{proof}
An argument similar to \eqref{eq:444} leads to
\begin{align*}
\seminormE{w_{\text{snap}}-w_{\text{off}}^{\text{H}}}{D}^2\leq 2\sum\limits_{i=1}^{N}\Big( H^{-2}\int_{\omega_i}\widetilde{\kappa}|u^{i}_{\text{snap}}-u^i_{\ell_i}|^2\dx+\int_{\omega_i}\chi_i^2\kappa
|\nabla (u^{i}_{\text{snap}}-u^i_{\ell_i})|^2\dx\Big).
\end{align*}
Together with Lemma \ref{lem:5.2}, we obtain
\begin{align*}
\seminormE{w_{\text{snap}}-w_{\text{off}}^{\text{H}}}{D}^2\leq 10H^{-2}\sum\limits_{i=1}^{N} \int_{\omega_i}\widetilde{\kappa}|u^{i}_{\text{snap}}-u^i_{\ell_i}|^2\dx.
\end{align*}
Then from Lemma \ref{lem:5.1}, we deduce
\begin{align*}
\seminormE{w_{\text{snap}}-w_{\text{off}}^{\text{H}}}{D}^2\leq 20 \max_{i=1,\cdots,N}\{{(H^{2}\lambda_{\ell_i+1}^{\hi})^{-1}}\}\sum\limits_{i=1}^{N}\Big(  {H}^2{\Cpoin{\omega_i}}\normLi{f}{\omega_i}^2+\seminormE{u_h}{\omega_i}^2\Big).
\end{align*}
Finally, the overlap condition \eqref{eq:overlap} together with \eqref{eq:uApriori} shows the desired assertion.
\end{proof}

Finally, we derive an error estimate for the CG approximation to Problem \eqref{eqn:pde} based
on the discrete POD multiscale space $V_{\text{off}}^{\text{H}}$.
\begin{proposition}\label{prop:Finalpod}
Assume that $f\in L^2_{{\kappa}^{-1}}(D)$ and $\ell_i\in \mathbb{N}_{+}$ for all $i=1,2,\cdots, N$. Let $u\in V$ and
$u_{{\rm off}}^{{\rm H}}\in V_{{\rm off}}^{{\rm H}}$ be the solutions to Problems \eqref{eqn:pde} and \eqref{cgvarform_pod}, respectively. Then there holds
\begin{align}\label{eq:podErr}
\seminormE{u-u_{{\rm off}}^{{\rm H}}}{D}&\leq \sqrt{2C_{{\rm ov}}}H\max_{i=1,\cdots,N}\Big\{ C_0H\Cpoin{\omega_i}+\sqrt{\Cpoin{\omega_i}}\Big\}\normLi{f}{D}\\
&+\sqrt{20C_{{\rm ov}}}\max_{i=1,\cdots,N}\Big\{{(H^{2}\lambda_{\ell_i+1}^{\hi})^{-\frac12}}\Big\} C_1\normLi{f}{D}
+\min\limits_{v_h\in V_h}\seminormE{u-v_h}{D}. \nonumber
\end{align}
\end{proposition}
\begin{proof}
This assertion follows from the Galerkin orthogonality property \cite[Corollary 2.5.10]{MR2373954}, the
triangle inequality and the fine-scale {\em a priori} estimate \eqref{eq:fineApriori}, Proposition
\ref{prop:FinalSnap} and Lemma \ref{lem:4.6}.
\end{proof}
\begin{remark}
Since the discrete eigenvalue problem \eqref{offeig} is generated from the continuous eigenvalue problem \eqref{eq:spectral} with finite ensembles $\{\phi_j^{\hi}\}_{j=1}^{L_i}$, a scaling argument shows
\begin{align*}
H^{2}\lambda_{n}^{\hi}\to \infty \quad\text{ as } n\to \infty \text{ and } h\to 0.
\end{align*}
This and \eqref{eq:podErr} imply the convergence of the POD solution $u_{\rm off}^{\rm H}$ in the energy norm.
\end{remark}

\section{Concluding remarks}\label{sec:conclusion}
In this paper, we have analyzed three types of multiscale methods in the framework of the generalized multiscale finite
element methods (GMsFEMs) for elliptic problems with heterogeneous high-contrast coefficients. Their convergence rates
 in the energy norm are derived under a very mild assumption on the source term, and are given in terms
of the eigenvalues and coarse grid mesh size. It is worth pointing out that the analysis does not rely on any oversampling
technique that is typically adopted in existing studies. The analysis indicates that the eigenvalue decay behavior of
eigenvalue problems with high-contrast heterogeneous coefficients is crucial for the convergence behavior
of the multiscale methods, including the GMsFEM. This motivates further investigations on such eigenvalue problems in order
to gain a better mathematical understanding of these methods. Some partial findings along this line have been presented
in the work \cite{li2017low}, however, much more work remains to be done.

\section*{Acknowledgements}
The work was partially supported by the Hausdorff Center for Mathematics, University of Bonn, Germany. The author acknowledges the
support from the Royal Society through a Newton international fellowship, and thanks Eric Chung (Chinese University of Hong Kong), Juan Galvis (Universidad Nacional de Colombia, Colombia), Michael Griebel (University of Bonn, Germany) and Daniel Peterseim (University of Augsburg, Germany) for fruitful discussions on the topic of the paper.

\appendix
\section{Very-weak solutions to boundary-value problems with high-contrast heterogeneous coefficients}
In this appendix, we derive a weighted $L^2$ estimate for boundary value problems with
high-contrast heterogeneous coefficients, which plays a crucial role in the error analysis.
Let Assumption \ref{ass:coeff} hold and let $\omega_i$ be a coarse neighborhood for any
$i=1,\cdots,N$. For any $g\in L^2(\partial\omega_i)$, we define the following elliptic problem
\begin{equation}\label{eq:pde-very}
\left\{\begin{aligned}
-\nabla\cdot(\kappa\nabla v)&=0 && \text{ in } \omega_i,\\
v&=g &&\text{ on }\partial \omega_i.
\end{aligned}\right.
\end{equation}
Our goal is to derive an weighted $L^2$ estimate of the solution $v$, which is independent of the high-contrast
in the coefficient $\kappa$. To this end, we employ a nonstandard variational form in the spirit
of the transposition method \cite{MR0350177}, and seek $v\in L^2(\omega_i)$ such that
\begin{align}\label{eq:nonstd-variational}
-\int_{\omega_i}v\nabla\cdot(\kappa\nabla z)\dx=-\int_{\partial \omega_i}g\kappa\frac{\partial z}{\partial n}\mathrm{d}s \quad\text{ for all }z\in X(\omega_i).
\end{align}
Here, $X(\omega_i)$ denotes the test space to be defined below. The main difficulty for our setting of piecewise
high-contrast coefficient is that the solution has only piecewise $H^2$ regularity, and thus, we cannot directly
apply the nonstandard variational form described above. The difficulty is overcome in Theorem \ref{thm:pw-Regularity}.

\begin{theorem}\label{lem:very-weak}
Assume that $\{\eta_j\}_{j=1}^{m}$ are of comparable magnitude and that $\etamaxmin{min}$ is sufficiently large. Let $g\in L^2(\omega_i)$ and let $v$ be the solution to \eqref{eq:pde-very}. Then there exists a constant ${\rm C}_{{\rm weak}}$ independent of the coefficient $\kappa$ such that
\[
\normLT{v}{\omega_i}\leq {\rm C}_{{\rm weak}}\|g\|_{L^2(\partial \omega_i)}.
\]
\end{theorem}

To prove it, we need a regularity result based on \cite{chu2010new,li2017low}.
\begin{theorem}\label{thm:pw-Regularity}
Assume that $\{\eta_j\}_{j=1}^{m}$ are of comparable magnitude and that $\etamaxmin{min}$ is sufficiently large. Let $w\in L^2_{\widetilde{\kappa}^{-1}}(\omega_i)$ and let $z\in H^1_{0}(\omega_i)$ be the unique solution to the following weak formulation
\begin{align}\label{eq:aux-z}
\forall q\in H^1_{0}(\omega_i): \int_{\omega_i}\kappa\nabla z\cdot\nabla q\;\dx=\int_{\omega_i}wq\;\dx.
\end{align}
Then for some constant ${\rm C}_{{\rm weak}} $ independent of the contrast, there holds
\begin{align*}
\|\eta_j\frac{\partial z}{\partial n}\|_{L^2(\partial\omega_i\cap D_j)}
&\leq\Cw \normLii{w}{\omega_i}\quad\text{ for all } j=0,1,\cdots,m.
\end{align*}
\end{theorem}
\begin{proof}
The triangle inequality, Poincar\'{e} inequality, and \cite[Eqn. (6.2) and Proposition 6.7]{li2017low} imply
\begin{equation}\label{eq:H1-estimate}
\begin{aligned}
\normHp{z}{\omega_i\cap D_0}{1}&\lesssim \Cpoin{\omega_i\cap D_0}\|w\|_{L^2(\omega_i)},\\
\normHp{z}{\omega_i\cap D_j}{1}&\lesssim \etamaxmin{min}^{-1}\Cpoin{\omega_i\cap D_0}\|w\|_{L^2(\omega_i)},\quad\text{ for } j=1,2,\cdots,m.
\end{aligned}
\end{equation}
Note that the $H^2$ seminorm regularity result in \cite[Theorem B.1]{chu2010new} does not depend on the distance between $\partial \omega_i$ and $D_j$ for any $j=1,\cdots,m$. Therefore, it can be extended to our situation directly:
\begin{align*}
|{z}|_{H^2(\omega_i\cap D_j)}&\lesssim \etamaxmin{min}^{-1}\|w\|_{L^2(\omega_i)} \quad\text{ for } j=0,1,\cdots,m .
\end{align*}
Combining the preceding two estimates and applying interpolation between $H^1(\omega_i)$ and $H^2(\omega_i)$ yield the $H^{3/2}(\omega_i)$ regularity estimate
\begin{align}\label{eq:h3/2}
\normHp{z}{\omega_i\cap D_j}{3/2}\lesssim \etamaxmin{min}^{-1}\|w\|_{L^2(\omega_i)}.
\end{align}
Furthermore, since $w\in L^2_{\widetilde{\kappa}^{-1}}(\omega_i)\subset L^2(\omega_i)$, by definition, we can obtain
\begin{align*}
\normL{w}{\omega_i}^2&=\int_{\omega_i}w^2\dx=\sum\limits_{j=0}^{m}\int_{\omega_i\cap D_j}
w^2\dx\nonumber\\
&\leq\sum\limits_{j=0}^{m}\eta_{j}\normLii{w}{\omega_i\cap D_j}^2
\lesssim \etamaxmin{min}\normLii{w}{\omega_i}^2.
\end{align*}
This, together with \eqref{eq:h3/2}, proves
\begin{align}\label{eq:h3/2ii}
\normHp{z}{\omega_i\cap D_j}{3/2}\lesssim \etamaxmin{min}^{-1/2}\normLii{w}{\omega_i}.
\end{align}
Since differentiation is continuous from $H^{3/2}(\omega_i)$ to $H^{1/2}(\omega_i)$, by the trace theorem, we have
\begin{align*}
\|\frac{\partial z}{\partial n}\|_{L^2(\partial\omega_i\cap D_j)}&\lesssim \|\frac{\partial z}{\partial n}\|_{H^{1/2}(\omega_i\cap D_j)}\lesssim \normHp{z}{\omega_i\cap D_j}{3/2},
\end{align*}
which, together with \eqref{eq:h3/2ii}, proves the desired assertion.
\end{proof}

Next we define a Lions-type variational formulation for Problem \eqref{eq:pde-very} when $\etamaxmin{min}$ is large
\cite[Section 6, Chapter 2]{MR0350177}.
To this end, let the test space $X(\omega_i)\subset H^1_{\kappa,0}(\omega_i)$ be defined by
\begin{align}\label{eq:test-space}
X(\omega_i):=\{z:-\nabla\cdot(\kappa\nabla z)\in L^2(\omega_i)\text{ and } z\in H^1_{\kappa,0}(\omega_i)\}.
\end{align}
This test space $X(\omega_i)$ is endowed with the norm $\|\cdot\|_{X(\omega_i)}$:
\[
\forall z\in X(\omega_i):\|z\|_{X(\omega_i)}^2=\int_{\omega_i}\kappa|\nabla z|^2\dx+\|\nabla\cdot(\kappa\nabla z)\|_{L^2(\omega_i)}^2.
\]
Below, we denote by $n_{i}(x)$ the unit outward normal (relative to $D_i$) to the interface
$\Gamma_i$ at the point $x\in \Gamma_i$. For a function $w$ defined on $\mathbb{R}^2\backslash
\Gamma_i$ for $i=1,2,\cdots,m$, we define for any $x\in \Gamma_i$,
\[
w(x)|_{\pm}:=\lim_{t\to 0^{+}} w(x\pm tn_{i}(x))\quad \text{ and }\quad
\frac{\partial}{\partial n_{i}^{\pm}}w(x):=\lim_{t\to 0^{+}}(\nabla w(x\pm tn_{i}(x))\cdot n_{i}(x))
\]
if the limit on the right hand side exists.
\begin{lemma}
Let $v$ be the solution to problem \eqref{eq:pde-very} and let the test space $X(\omega_i)$ be defined in \eqref{eq:test-space}. Then the nonstandard variational form \eqref{eq:nonstd-variational} is well posed.
\end{lemma}
\begin{proof}
For all $z\in X(\omega_i)$, let $w:=-\nabla\cdot(\kappa\nabla z)$, then by definition, $w\in L^2(\omega_i)$.
Recall the continuity of the flux implied by the definition, i.e.,
\begin{align}\label{eq:flux}
\forall z\in X(\omega_i):\quad\eta_j\frac{\partial z}{\partial n_j^{-}}=\frac{\partial z}{\partial n_j^{+}}\quad\text{ for all } j=1,\cdots,m.
\end{align}
For all $z\in X(\omega_i)$, we obtain
\begin{align*}
\int_{\omega_i}-\nabla\cdot(\kappa\nabla v)z\;\dx&=\sum_{j=0}^{m}\int_{\omega_i\cap D_j}-\nabla\cdot(\kappa\nabla v)z\;\dx
=\sum_{j=0}^{m}\int_{\omega_i\cap D_j}\Big(-\nabla\cdot(\kappa\nabla v z)+\kappa\nabla z\cdot\nabla v\Big)\;\dx\\
&=\Big(\int_{\partial D_0\backslash\partial\omega_i}\kappa\frac{\partial v}{\partial n_{j}^{+}} z\mathrm{d}s-\sum_{j=1}^{m}\int_{\partial D_j\backslash\partial\omega_i}\kappa\frac{\partial v}{\partial n_{j}^{-}} z\mathrm{d}s\Big)+\sum_{j=0}^{m}\int_{\omega_i\cap D_j}\kappa\nabla z\cdot\nabla v\;\dx.
\end{align*}
The continuity of the flux for $v$ shows that the sum of the first two terms vanishes. We apply the divergence
theorem again, together with the continuity of flux for $z$, and derive
\begin{align*}
\int_{\omega_i}-\nabla\cdot(\kappa\nabla v)z\;\dx&=
\sum_{j=0}^{m}\int_{\omega_i\cap D_j}\kappa\nabla z\cdot\nabla v\;\dx
=\sum_{j=0}^{m}\int_{\omega_i\cap D_j}\nabla\cdot(\kappa\nabla z v)-\nabla\cdot(\kappa\nabla z)v\;\dx\\
&=\Big(-\int_{\partial D_0\backslash\partial\omega_i}\kappa\frac{\partial z}{\partial n_{j}^{+}} v\mathrm{d}s+\sum_{j=1}^{m}\int_{\partial D_j\backslash\partial\omega_i}\kappa\frac{\partial z}{\partial n_{j}^{-}} v\mathrm{d}s\Big)
+\int_{\partial\omega_i}\kappa\frac{\partial z}{\partial n} g\mathrm{d}s\\
&-\int_{\omega_i}\nabla\cdot(\kappa\nabla z)v\;\dx.
\end{align*}
The continuity of flux \eqref{eq:flux} indicates that the first term vanishes, and this proves \eqref{eq:nonstd-variational}.

To prove the well-posedness of the nonstandard variational form \eqref{eq:nonstd-variational}, we introduce a bilinear form
$c(\cdot,\cdot)$ on $L^2_{\widetilde{\kappa}}(\omega_i)\times L^2_{\widetilde{\kappa}^{-1}}(\omega_i)$ and a linear form $b(\cdot)$ on $L^2_{\widetilde{\kappa}^{-1}}(\omega_i)$, defined by
\begin{align*}
c(w_1,w_2)&:=\int_{\omega_i}w_1 w_2\;\dx\quad\text{ for all } w_1\in L^2_{\widetilde{\kappa}}(\omega_i)\text{ and }w_2\in L^2_{\widetilde{\kappa}^{-1}}(\omega_i),\\
b(w)&:=\int_{\partial\omega_i}\kappa\frac{\partial z}{\partial n} g\;\mathrm{d}s\quad
\text{ for all }w\in L^2_{\widetilde{\kappa}^{-1}}(\omega_i),
\end{align*}
with $z$ being the unique solution to \eqref{eq:aux-z}.
It follows from Theorem \ref{thm:pw-Regularity} that
\begin{align}\label{eq:b}
\|b\|:=\sup_{w\in L^2_{\widetilde{\kappa}^{-1}}(\omega_i)}\frac{b(w)}{\|w\|_{L^2_{\widetilde{\kappa}^{-1}}(\omega_i)}}\leq \text{C}_{\text{weak}}\|g\|_{L^2(\partial\omega_i)}.
\end{align}
This implies that $b$ lies in the dual space of $L^2_{\widetilde{\kappa}^{-1}}(\omega_i)$.
Since the dual space of $L_{\widetilde{\kappa}^{-1}}^2(\omega_i)$
is $L^2_{\widetilde{\kappa}}(\omega_i)$, cf. Remark \ref{rem:dual}, this
yields well-posedness of the following variational problem: find $v\in L^2_{\widetilde{\kappa}}(\omega_i)$ such that
\begin{align}\label{eq:variation2}
c(v,w)=b(w)\quad\text{ for all }w\in L^2_{\widetilde{\kappa}^{-1}}(\omega_i).
\end{align}
The equivalence of problems \eqref{eq:variation2} and \eqref{eq:nonstd-variational} implies the desired well-posedness of \eqref{eq:nonstd-variational}.
\end{proof}
Finally, we are ready to prove Theorem \ref{lem:very-weak}.
\begin{proof}[Proof of Theorem \ref{lem:very-weak}]
For all $w\in L^2_{\widetilde{\kappa}^{-1}}(\omega_i)$, we obtain from \eqref{eq:b} and \eqref{eq:variation2}
\begin{align*}
\int_{\omega_i} v w\;\dx:=c(v,w)=b(w)\leq \text{C}_{\text{weak}}\normLii{w}{\omega_i}\|g\|_{\partial\omega_i}.
\end{align*}
Since $(L_{\widetilde{\kappa}^{-1}}^2(\omega_i))^*=L^2_{\widetilde{\kappa}}(\omega_i)$, cf. Remark \ref{rem:dual},
we get the desired assertion. This completes the proof.
\end{proof}
\bibliographystyle{abbrv}
\bibliography{reference}

\end{document}